\documentclass[11pt,twoside]{amsart}
\usepackage{amssymb}
\usepackage{graphicx,color}
\usepackage{enumerate}

\usepackage[all]{xy}
\usepackage{tikz}
\usetikzlibrary{matrix}

\newtheorem{theorem}{Theorem}[section]
\newtheorem{proposition}[theorem]{Proposition}
\newtheorem{lemma}[theorem]{Lemma}
\newtheorem{corollary}[theorem]{Corollary}
\newtheorem{conjecture}[theorem]{Conjecture}
\newtheorem{problem}[theorem]{Problem}
\newtheorem{question}[theorem]{Question}

\theoremstyle{definition}
\newtheorem{definition}[theorem]{Definition}
\newtheorem{notation}[theorem]{Notation}
\newtheorem{example}[theorem]{Example}
\newtheorem{remark}[theorem]{Remark}

\newcommand{\skipit}[1]{{}}
\newcommand{\prfend}{\hbox to7pt{\hfil}
\par\vskip-\baselineskip\hbox to\hsize
{\hfil\vbox {\hrule width6pt height6pt}}\vskip\baselineskip}

\newcommand{\ZZ}{\mathbb{Z}}

\newcommand {\PP}{\mathbb{P}}

\newcommand{\cM}{\mathcal{M}}

\newcommand{\cB}{\mathcal{B}}
\newcommand{\W}{\mathcal{W}_{d}}

\newcommand{\kk}{\mathbb{K}}

\DeclareMathOperator{\HF}{HF}

\newcommand{\bG}{\overline{G}}
\DeclareMathOperator{\reg}{reg}
\DeclareMathOperator{\PV}{PV}

\DeclareMathOperator{\GL}{GL}
\DeclareMathOperator{\diag}{diag}
\DeclareMathOperator{\GCD}{GCD}

\DeclareMathOperator{\I}{I}
\DeclareMathOperator{\A}{A}

\DeclareMathOperator{\supp}{supp}
\DeclareMathOperator{\codim}{codim}
\DeclareMathOperator{\LT}{LT}

\newcommand{\myarrow}[2]{\hbox to #1pt{\hfil$\to$\hfil}{\hskip-#1pt{\raise
10pt\hbox to#1pt{\hfil$\scriptscriptstyle #2$\hfil}}}}

\title[Monomial projections of Veronese varieties]
{Monomial projections of Veronese varieties: \\
New results and conjectures}

 \author[L. Colarte-G\'omez]{Liena Colarte-G\'omez} 
 \address{Institute of Mathematics of the Polish Academy of Sciences, \'Sniadeckich 8, 00-656 Warwaw, Poland} \email{lcolartegomez@impan.pl}

\author[R. M. Mir\'o-Roig]{Rosa M.\ Mir\'o-Roig} 
  \address{Facultat de
  Matem\`atiques i Inform\`atica, Universitat de Barcelona, Gran Via des les
  Corts Catalanes 585, 08007 Barcelona, Spain} \email{miro@ub.edu, ORCID 0000-0003-1375-6547}

\author[L. Nicklasson]{Lisa Nicklasson}
\address{Dipartimento di Matematica, Università di Genova, Via Dodecaneso 35,
16146 Genova, Italy}
\email{nicklasson@dima.unige.it}
\thanks{\hspace{-15pt} The first author was partially supported by the grant {\em Borsa Ferran Sunyer i Balaguer 2021}. \\ The first and second authors have been  partially supported by the grant PID2019-104844GB-I00.\\
The third author was supported by the grant KAW 2019.0512} 

\begin{document}
\begin{abstract} 
In this paper, we consider the homogeneous coordinate rings $\A(Y_{n,d}) \cong \kk[\Omega_{n,d}]$ of monomial projections $Y_{n,d}$ of Veronese varieties parameterized by subsets $\Omega_{n,d}$ of monomials of degree $d$ in $n+1$ variables where: (1) $\Omega_{n,d}$ contains all monomials supported in at most $s$ variables and, (2) $\Omega_{n,d}$ is a set of monomial invariants of a finite diagonal abelian group $G \subset \GL(n+1,\kk)$ of order $d$. Our goal is to study when $\kk[\Omega_{n,d}]$ is a quadratic algebra and, if so, when $\kk[\Omega_{n,d}]$ is  Koszul or G-quadratic. For the family (1), we prove that $\kk[\Omega_{n,d}]$ is quadratic when $s \ge \lceil \frac{n+2}{2} \rceil$. For the family (2), we completely characterize when $\kk[\Omega_{2,d}]$ is quadratic in terms of the group $G \subset \GL(3,\kk)$, and we prove that $\kk[\Omega_{2,d}]$ is quadratic if and only if it is Koszul. We also provide large families of examples where $\kk[\Omega_{n,d}]$ is G-quadratic. 
\end{abstract}
\maketitle

\section{Introduction}
The \emph{Veronese variety} $X_{n,d} \subset \PP^{\binom{n+d}{n}-1}$ is the projective toric variety parameterized by the set $\cM_{n,d}$ of all monomials of degree $d$ in $R = \kk[x_0,\hdots,x_n]$. {\em Monomial projections} $Y_{n,d}$ of Veronese varieties $X_{n,d}$ are projective toric varieties parameterized by subsets $\Omega_{n,d} \subset \cM_{n,d}$ of monomials. Their homogeneous coordinate rings $\A(Y_{n,d})$ are isomorphic to the monomial $\kk$-algebras $\kk[\Omega_{n,d}]$. Many geometric properties of $Y_{n,d}$ can be explained by means of the algebraic structure of $\kk[\Omega_{n,d}]$, and vice versa. For this reason, monomial projections of Veronese varieties appear at the crossroads between commutative algebra and algebraic geometry.

In order to better understand the properties of monomial projections $Y_{n,d}$ of Veronese varieties, and the monomial $\kk$-algebras $\kk[\Omega_{n,d}]$, it is useful to study their homogeneous ideals $\I(Y_{n,d})$. Moreover, these ideals are of great interest since $\I(Y_{n,d})$ is the homogeneous binomial prime ideal of syzygies of the monomials in $\kk[\Omega_{n,d}]$. Among them, the ideals $\I(Y_{n,d})$ generated by quadrics are in the spotlight, especially since many of them provide examples of \emph{Koszul algebras}. Koszul algebras were introduced by Priddy in \cite{Pr}, and since then, they have played an important role in the research fields of commutative algebra and algebraic geometry. We recall that, when $\I(Y_{n,d})$ is generated by quadrics, the algebra $\kk[\Omega_{n,d}]$ is called {\em quadratic}, and when  $\I(Y_{n,d})$ admits a Gr\"obner basis of quadrics, $\kk[\Omega_{n,d}]$ is said to be {\em G-quadratic}. It is known that Koszul algebras are quadratic and that G-quadratic algebras are Koszul. The converses 
 are in general false, see for instance \cite{ERT}. 
 
 One of the most extensively studied examples is precisely the Veronese variety $X_{n,d}$, whose associated monomial $\kk$-algebra $\kk[\cM_{n,d}]$ is G-quadratic and, hence, Koszul and quadratic (\cite{Backelin-Froberg, ERT, Groebner}). Unlike the Veronese variety $X_{n,d}$, the homogeneous ideal $\I(Y_{n,d})$ of a monomial projection $Y_{n,d}$ is in general not generated by quadrics, and binomials of arbitrarily high degree are often required (see Example \ref{Example: higher degrees binomial}). 
In this work, we focus on two families of monomial subsets, namely (1) $\Omega_{n,d}$ containing at least all monomials supported in at most $s$ variables, and (2) $\Omega_{n,d}$ being the set of monomial invariants of a finite  abelian group $G \subset \GL(n+1,\kk)$ of order $d$. Our main objective is to investigate which monomial subsets $\Omega_{n,d}$ from these two families give rise to quadratic, G-quadratic, and Koszul algebras. Along the way, we pose several questions, problems and conjectures.

Family (1) is motivated by the fact that it contains the {\em pinched Veronese algebras} $\PV(n,d,s)$,  i.\,e.\ the $\kk$-algebra generated by all monomials in $R$ of degree $d$ in at most $s$ variables. The problem of determining when $\PV(n,d,s)$ is a quadratic or Koszul algebra was posed in \cite{Conca-DeNegri-Rossi}. In the main result of this paper (Theorem \ref{thm:pv_quadratic}), we contribute to this problem by proving that the monomial $\kk$-algebras given by (1) are quadratic when $s \geq \lceil \frac{n+2}{2} \rceil$.
A key observation is that the $\kk$-vector spaces $\langle \Omega_{n,d} \rangle$ are 2-normal for these values of $s$, i.\,e.\ $\kk[\Omega_{n,d}]_2 = R_{2d}$. Family (2) has been studied in \cite{ThesisLiena, CMM-R, CM-R, CM-R1}, where the authors proved that the associated ideals $\I(Y_{n,d})$ are generated by binomials of degree at most $3$. In this direction, for $n = 2$ we completely determine when $\kk[\Omega_{2,d}]$ is a quadratic algebra in terms of the finite abelian group $G \subset \GL(3,\kk)$ and,  furthermore, we prove that $\kk[\Omega_{2,d}]$ is a quadratic algebra if and only if it is a Koszul algebra (Theorem \ref{thm:Koszul_surf} and Propositions \ref{prop:Koszul_surf} and \ref{prop:G_koszul}). For those being Koszul, we provide families of examples where $\kk[\Omega_{2,d}]$ is a G-quadratic algebra (Propositions \ref{Proposition: G basis 0,1,k} and \ref{prop:G-quadratic_group}), and this allows us to produce further examples in any dimension (Proposition \ref{prop:GB_extend_groups} and Corollary \ref{coro:GB_extend_groups}). 

\vspace{0.15cm} This work is organized as follows. In Section \ref{Notation and preliminaries}, we gather the basic notions and definitions needed in the body of this paper. In Section \ref{Section: gen mono proj}, we begin our study of homogeneous ideals of monomial projections $Y_{n,d}$ of Veronese varieties. A first important observation is a combinatorial characterization of when a binomial in $\I(Y_{n,d})$ is generated by binomials of lower degree. In Subsection \ref{subsec:monomial_2-normal}, we focus our attention on family (1). We 
 show that if a subset of monomials $\Omega_{n,d}$ spans a $2$-normal vector space, then $\I(Y_{n,d})$ is generated by binomials of degree at most $3$ (Proposition \ref{Proposition: generation at most 3}), and we exhibit an example where $\I(Y_{n,d})$ is minimally generated by quadrics and cubics. Next, we characterize when a subset $\Omega_{n,d}$ in (1) spans a $2$-normal vector space (Proposition \ref{Propsition: 2-normality}), and we use this fact in the proof of our main result: a monomial $\kk$-algebra $\kk[\Omega_{n,d}]$ with $\Omega_{n,d}$ as in (1) is a quadratic algebra when $s \ge \lceil \frac{n+2}{2} \rceil$ (Theorem \ref{thm:pv_quadratic}). 
In Subsection \ref{subsec:monomials_group_inv}, we introduce the family (2) and we gather what is known about their homogeneous ideals.

 Section \ref{Section: Koszulness} is devoted to the problem of determining when quadratic algebras $\kk[\Omega_{n,d}]$ are Koszul and further G-quadratic. We focus mainly on monomial projections of the Veronese surface of family (2). Given a finite abelian group $G \subset \GL(3,\kk)$ of order $d$, let $\Omega_{2,d}$ be the set of all monomial invariants of $G$ of degree $d$. Using the results of \cite{ThesisLiena} and \cite{Conca-Iyengar-Nguyen-Romer}, we completely characterize when $\kk[\Omega_{2,d}]$ is quadratic in terms of the finite abelian group $G$, and we establish that $\kk[\Omega_{2,d}]$ is a quadratic algebra if and only if it is a Koszul algebra (Theorem \ref{thm:Koszul_surf} and Propositions \ref{prop:Koszul_surf} and \ref{prop:G_koszul}). Afterwards, for this subfamily we ask if being quadratic is also equivalent to being G-quadratic. We establish this fact for large subfamilies of algebras $\kk[\Omega_{2,d}]$ with $G \subset \GL(3,\kk)$ a finite cyclic group (Propositions \ref{Proposition: G basis 0,1,k} and \ref{prop:G-quadratic_group}). To tackle these topics more in general, we give a technique (Lemma \ref{lemma:GB_extension}) to produce examples of G-quadratic algebras $\kk[\Omega_{N,d}]$ with $N \geq 3$ by means of G-quadratic algebras $\kk[\Omega_{n,d}]$ with $N > n$. In particular, we apply this technique to G-quadratic algebras $\kk[\Omega_{n,d}]$ in the family (2) (Proposition \ref{prop:GB_extend_groups} and Corollary \ref{coro:GB_extend_groups}). 

Finally, in Section \ref{Section: conjectures}, we gather the main questions and problems posed along Section \ref{Section: Koszulness}, and we present our two main related conjectures regarding the family (2) of monomial projections.

\vskip 4mm \noindent
 {\bf Acknowledgements.} The authors are grateful to the anonymous referees whose comments and suggestions help to improve the exposition of this paper. This work was partially developed while the first and second authors were visiting the {\em Universit\`{a} degli Studi di Genova}, they would like to thank the mathematical community of the {\em Dipartimento di Matematica (DIMA)} for its hospitality. 

\section{Notation and preliminaries}\label{Notation and preliminaries}
Throughout this note, $R = \kk[x_0,\hdots, x_n]$ will be a polynomial ring over a field $\kk$, and $R^{\Lambda}$ the subring of invariants of a finite subgroup $\Lambda \subset \GL(n+1,\kk)$.
Given positive integers $n$ and $d$, we denote by $\cM_{n,d} = \{x_0^{a_0}\cdots x_n^{a_{n}} \in R \mid
a_{0}+\cdots + a_{n} = d\}$ the set of monomials of degree $d$ in $R$, and its cardinality by $ N_{n,d} =|\cM_{n,d}|= \binom{n+d}{n}$.

Recall that an algebra $A$ is a standard graded $\kk$-algebra if $A =\bigoplus _{i \ge 0} A_i$ with $A_0=\kk$,
and $A_1$ is finite dimensional as a $\kk$-space and generates $A$ as a $\kk$-algebra. 
Equivalently, $A$ can be presented as $A = R/I$ where $I \subset R$ is a homogeneous ideal.

\begin{definition} Let $A$ be a standard graded $\kk$-algebra.

\begin{itemize}
\item[(i)] The algebra $A$ is said to be \emph{quadratic} if its defining ideal $I$ is generated by quadrics.
\item[(ii)] The algebra $A$ is said to be \emph{G-quadratic} if $I$ has a Gr\"obner
basis of quadrics (with respect to some coordinate system of $R$ and some term order).
\item[(iii)] The algebra $A$ is said to be \emph{Koszul} if $\reg _A \kk=0$, i.\,e.\ $\kk $ has a linear $A$-resolution.
\end{itemize}
\end{definition}

It is known that Koszul algebras are quadratic and that G-quadratic algebras are Koszul, as mentioned before the converses are not true in general.

We let $X_{n,d} \subset \PP^{N_{n,d}-1}$ denote the \emph{Veronese variety}, defined as the $n$-dimensional projective variety parameterized by $\cM_{n,d}= \{m_1, \ldots, m_{N_{n,d}}\}$,  i.\,e.\ the image of the Veronese embedding
\[\nu_{n,d}: \PP^{n} \longrightarrow \PP^{N_{n,d}-1}, \quad p  \to (m_1(p):\cdots:m_{N_{n,d}}(p)).\]
 We take a new polynomial ring $S= \kk[w_1,\hdots, w_{N_{n,d}}]$ in $N_{n,d}$ variables. Then we have an isomorphism $\kk[\cM_{n,d}] \cong S/ \I(X_{n,d})$ between the \emph{$d$-th Veronese subalgebra} $\kk[\cM_{n,d}] \subset R$ and the homogeneous coordinate ring of $X_{n,d}$. The homogeneous prime ideal $\I(X_{n,d})$ is minimally generated by the quadratic binomials
\begin{equation*}
\{ w_iw_j-w_kw_\ell \ | \ m_im_j=m_km_\ell, \ \ 1 \le i,j,k,\ell \le N_{n,d} \}.
\end{equation*}

 Given a subset $\Omega_{n,d} \subseteq \cM_{n,d}$ of $\mu_{n,d} \leq N_{n,d}$ monomials, we denote by $\varphi_{\Omega_{n,d}}: \PP^{n} \dashrightarrow \PP^{\mu_{n,d}-1}$ the rational map defined by $\Omega_{n,d}$ and we say that $Y_{n,d} := \overline{\varphi_{\Omega_{n,d}}(\PP^{n})} \subset \PP^{\mu_{n,d}-1}$ is the {\em monomial projection of the Veronese variety $X_{n,d}$ parameterized by $\Omega_{n,d}$.}  
 We have the commutative diagram

\begin{center}
\begin{tikzpicture}
\matrix (M) [matrix of math nodes, row sep=3em, column sep=3em, nodes={minimum height = 1cm, minimum width = 1cm, anchor=center}]{
\PP^{n} & X_{n,d} \\
        & Y_{n,d} \\
};
\draw[->] (M-1-1) -- (M-1-2) node[midway,above]{$\nu_{n,d}$};
\draw[->,dashed] (M-1-2) -- (M-2-2) node[midway,right]{$\pi$};
\draw[->,dashed] (M-1-1) -- (M-2-2) node[midway,below]{$\varphi_{\Omega_{n,d}}\hspace{0.5cm}$ };
\end{tikzpicture}
\end{center}
where $\pi$ is the projection of the Veronese variety $X_{n,d} \subset \PP^{N_{n,d}-1}$ from the linear subspace generated by the coordinate points $(0:\cdots:0:1:0:\cdots:0) \in \PP^{N_{n,d}-1}$, with $1$ in position $i$ such that $m_i \notin \Omega_{n,d}$, to the linear subspace $V(w_{m_{i}}, \; m_{i} \notin \Omega_{n,d}) \subset \PP^{N_{n,d}-1}$.

We let $\A(Y_{n,d})$ denote the homogeneous coordinate ring $S/\I(Y_{n,d})$ of $Y_{n,d}$, where $S=\kk[w_1, \ldots, w_{\mu_{n,d}}]$. The homogeneous ideal $\I(Y_{n,d})$ can be defined as the kernel of the homomorphism $\rho: S \longrightarrow R$ defined by $\rho(w_i)=m_{i}$.
Here, $m_1, \ldots, m_{\mu_{n,d}}$ are the monomials of $\Omega_{n,d}$, after renaming the monomials in $\cM_{n,d}$ if necessary. Hence, we have an isomorphism between $\A(Y_{n,d})$ and the monomial $\kk$-algebra $\kk[\Omega_{n,d}]$ generated by the monomials in $\Omega_{n,d}$; we will use them indistinguishably in the sequel. 
The prime ideal $\ker \rho = \I(Y_{n,d})$ is generated by binomials 
\begin{equation}\label{Equation: set of binomials}
\left\{ \prod_{i=1}^{\mu_{n,d}} w_{i}^{\alpha_i} - \prod_{i=1}^{\mu_{n,d}} w_i^{\beta_i} \ \Bigg| \   \prod_{i=1}^{\mu_{n,d}} m_i^{\alpha_i} = \prod_{i=1}^{\mu_{n,d}} m_i^{\beta_i}\right\}.
\end{equation}

For a monomial $m=x_0^{a_0} \cdots x_n^{a_n} \in R$, the \emph{support of $m$} is the set $\supp(m) = \{x_i \ | \ a_i>0\}$. In this paper, we concentrate our attention on two families of monomial projections parameterized by 
\begin{enumerate}
    \item $\Omega_{n,d} \supset  \{ m \in \cM_{n,d} \ | \ |\supp(m)| \le s\}$ for a given number $s \ge \lceil \frac{n+2}{2} \rceil$,
    \item $\Omega_{n,d}=\{ m \in \cM_{n,d} \ | \ m \ \text{is an invariant of a finite abelian group} \ G \subset \GL(n+1,\kk)\}$.
\end{enumerate}
A more detailed description of the second family is provided in Subsection \ref{subsec:monomials_group_inv}. Our first task is to study minimal generating sets of the homogeneous binomial ideals arising from these two families. 

\section{Generators of monomial projection of Veronese varieties}
\label{Section: gen mono proj}
While the homogeneous ideal $\I(X_{n,d})$ of the Veronese variety is generated by quadrics \cite{Groebner}, in general, the ideal of a monomial projection of the Veronese variety may require generators of arbitrarily high degree. Let us illustrate this well-known fact with a few examples. 

\begin{example}\label{Example: higher degrees binomial}\rm
We take $n = 2$, $d \ge 4$, and let $Y_{2,d}$ be the monomial projection of the Veronese surface $X_{2,d}$ parameterized by
\[\Omega_{2,d} = \left\{x_0^d,x_1^d,x_0x_1^{d-2}x_2, x_0^2x_1^{d-4}x_2^2,x_2^d, x_0^{d-1}x_1\right\}.\]
Using Macaulay2 \cite{Macaulay2}, we check that, for $d=4$, the homogeneous ideal $\I(Y_{2,4})$ of $Y_{2,4}$ is minimally generated by two binomials of degree $2$ and one binomial of degree $4$. For $d=5$, the ideal $\I(Y_{2,5})$ is minimally generated by one binomial of degree $2$, two binomials of degree $3$, and one binomial of degree $5$. If we take $d=6$,  then $\I(Y_{2,6})$ is minimally generated by four binomials of degree $2$ and one binomial of degree $6$.
\end{example}

 From now on, we fix positive integers $n$ and $d$, and a subset $\Omega_{n,d} \subset \cM_{n,d}$ of $\mu_{n,d}=|\Omega_{n,d}|$ monomials. The monomial projection parameterized by $\Omega_{n,d}$ is denoted by $Y_{n,d}$.
We denote by $\I(Y_{n,d})_{k}$ the set of binomials of degree $k$ appearing in (\ref{Equation: set of binomials}), and we refer to these binomials as \emph{$k$-binomials}. We say that a $k$-binomial is \emph{trivial} if it is a $(k-1)$-binomial multiplied by a variable.

\begin{definition}\label{Definition: Ik-sequence} \rm Let $k \geq 3$ be an integer and $w^{\alpha} = w^{\alpha_{+}}- w^{\alpha_{-}}\in \I(Y_{n,d})_k$ a non-trivial $k$-binomial, where $w^{\alpha_{+}}$ and $w^{\alpha_{-}}$ denotes monomials with exponent vectors $\alpha_{+}, \alpha_{-} \in \ZZ_{\geq 0}^{\mu_{n,d}}$. An {\em $\I(Y_{n,d})_{k}$-sequence from $w^{\alpha_{+}}$ to $w^{\alpha_{-}}$} is a finite sequence $(w^{(1)}, \hdots, w^{(t)})$ of monomials of $S$ of degree $k$ satisfying the two conditions:
\begin{enumerate}[(i)] 
\item $w^{(1)} = w^{\alpha_{+}}$ and $w^{(t)} = w^{\alpha_{-}}$,
\item for all $1\le j<t$,  $w^{(j)}-w^{(j+1)}$ is a trivial $k$-binomial.
\end{enumerate}
\end{definition}

In the following proposition, we gather a well-known fact about  binomial ideals. It characterizes when a binomial in $\I(Y_{n,d})_{k}$ belongs to the ideal $(\I(Y_{n,d})_{k-1})$ and, hence, it is not needed in a minimal generating set of $\I(Y_{n,d})$. For sake of completeness, we include a simple proof. 

\begin{proposition}\label{Proposition: Criterion generation} Let $w^{\alpha} = w^{\alpha_{+}}-w^{\alpha_{-}} \in \I(Y_{n,d})_k$ with $k \ge 3$. The following are equivalent.
\begin{enumerate}[(i)]
\item $w^{\alpha} \in (\I(Y_{n,d})_{k-1})$.
\item  There is an $\I(Y_{n,d})_{k}$-sequence from $w^{\alpha_{+}}$ to $w^{\alpha_{-}}$.
\end{enumerate}
\end{proposition}
\begin{proof} It is clear that (ii) implies (i).
Let us prove the converse. Assume that $w^{\alpha} \in (\I(Y_{n,d})_{k-1})$. We denote by $ \{q_{1},\hdots, q_{N}\}$ with $q_{j} = q_{j}^+ - q_{j}^-$ a system of binomial generators of $\I(Y_{n,d})_{k-1}$. By hypothesis, there exist linear forms $\{\ell_{1},\hdots, \ell_{N}\} \subset S$ such that $w^{\alpha_+} = \ell_{1}q_{1} + \cdots + \ell_{N}q_{N}+w^{\alpha_{-}}$. We write $\ell_j=\sum_{i=0}^{N_{n,d}-1}a_{ij}w_i$, where each $a_{ij} \in \kk$.  Therefore, 
$$w^{\alpha_{+}} = \sum_{j=1}^N\sum_{i=0}^{N_{n,d}-1} (a_{ij}w_{i}q_{j}^+
-  a_{ij}w_{i}q_{j}^-) + w^{\alpha_{-}}.$$ 
Hence, there exists a pair $i_0,j_0$ such that $a_{i_0j_0}=1$ and
 $w^{\alpha_{+}} = w_{i_0}q_{j_0}^+$,
 or $a_{i_0j_0}=-1$ and $w^{\alpha_{+}} = w_{i_0}q_{j_0}^-$. Assume $a^{j_0}_{i_0}=1$  (analogously we deal with the case $a^{j_0}_{i_0}=-1$). Set
  $w^{(2)} = w_{i_0}q_{j_0}^-$. We have
  $$w^{\alpha_{+}} =w^{\alpha_{+}} - w^{(2)}+\sum _{(i,j)\neq (i_0,j_0)}(a_{ij}w_{i}q_{j}^+
-  a_{ij}w_{i}q_{j}^-) + w^{\alpha_{-}}.$$
  Thus,
  $$w^{(2)}=\sum _{(i,j)\ne (i_0,j_0)}(a_{ij}w_{i}q_{j}^+
-  a_{ij}w_{i}q_{j}^-) + w^{\alpha_{-}}.$$
Now we can repeat the same argument for $w^{(2)}$. Considering that the number of summands decreases in each step, we can assure that we end with what we are looking for, namely a $\I(Y_{n,d})_k$-sequence from $w^{\alpha_+}$ to $w^{(t)}=w^{\alpha_-}$.
\end{proof}

\subsection{Monomial projections and 2-normality}\label{subsec:monomial_2-normal}

\begin{definition}\label{Definition: 2-normal} 
A vector subspace $V \subseteq R_d$ is said to be \emph{$2$-normal} if $\kk[V]_2 = R_{2d}$.
\end{definition}
 
$2$-normality was introduced in \cite{Birkenhake}, as an extension of the classical notion of projective normality, to study the minimal free resolution of $\kk[V]$ using Koszul cohomology methods. Here, the vector space $V$ is projectively normal if $V = R_d$. In these terms, any proper subset $\Omega_{n,d} \subsetneq \cM_{n,d}$ of monomials generates a non-projectively normal subspace.
 Notice that, for $2$-normal vector subspaces $V \subset R_d$, it holds that $\kk[V]_t = R_{td}$, for all $t \geq 2$.

In this subsection, we focus on subsets $\Omega_{n,d}$ of monomials spanning $2$-normal subspaces.  We shall see that subsets in family (1) $\Omega_{n,d} \supset  \{ m \in \cM_{n,d} \ | \ |\supp(m)| \le s\}$, for a given number $s \ge \lceil \frac{n+2}{2} \rceil$, do indeed span $2$-normal subspaces. Next, we apply $2$-normality to bound the degree of minimal generators for the binomial ideals $\I(Y_{n,d})$ associated to monomial projections $Y_{n,d}$. Our main result shows that $\I(Y_{n,d})$ is minimally generated by binomials of degree $2$ (Theorem \ref{thm:pv_quadratic}).

\begin{proposition}\label{Propsition: 2-normality} Let $\Omega_{n,d} \subset \cM_{n,d}$ be a subset of monomials. The $\kk$-vector space $\langle \Omega_{n,d} \rangle$ is $2$-normal in the following cases:
\begin{itemize}
\item[(i)] $\Omega_{n,d} \supseteq \{ m \in \cM_{n,d} \ | \ |\supp(m)| \le \lceil(n+2)/2 \rceil \},$
\item[(ii)] $\Omega_{n,d} = \cM_{n,d}\setminus \{m\}$ with $m \in \cM_{n,d} \setminus \{x_0^{d},\hdots, x_{n}^{d}\}$. 
\end{itemize}
That is, the Hilbert function of the monomial projection $Y_{n,d}$ parameterized by $\Omega_{n,d}$ satisfies $\HF(Y_{n,d},k) = \binom{n+kd}{n}$ for all $k \geq 2$. 
\end{proposition}

\begin{proof} (i) Let $m=x_0^{a_0}x_1^{a_1} \cdots x_n^{a_n}$ be a monomial of degree $2d$. We shall prove that $m$ is the product of two monomials of degree $d$ and support of size at most $\lceil(n+2)/2 \rceil$. After possibly reindexing our variables, we can assume that $a_0 \le a_1 \le \dots \le a_n$. Let
\[
m_1= \prod_{1 \le 2k+1 <n} x_{2k+1}^{a_{2k+1}} \quad \mbox{and} \quad m_2=\prod_{0 \le 2k <n}x_{2k}^{a_{2k}}
\]
so that $m=m_1m_2x_n^{a_n}$. Notice that the supports of $m_1$ and $m_2$ are of size at most $\lceil n/2 \rceil$. The degree of $m_1$ is at most $d$ and
\[
2 \deg(m_1)=\sum_{1 \le 2k+1 <n} 2a_{2k+1} \le \sum_{1 \le 2k+1 <n}(a_{2k+1}+a_{2k+2}) \le \sum_{i=1}^n a_i \le 2d.
\]
Similarly,
\[
2 \deg(m_2)=\sum_{0 \le 2k <n} 2a_{2k}  \le \sum_{0 \le 2k <n}(a_{2k}+a_{2k+1}) \le \sum_{i=0}^n a_i =2d, 
\] 
so $\deg(m_2) \le d$. Now we can write 
\begin{equation}\label{eq:2normal_decomp}
m=(m_1x_n^{d-\deg(m_1)})(m_2x_n^{d-\deg(m_2)}).
\end{equation}
\noindent (ii) Let $m \in \cM_{n,d}\setminus \{x_0^{d},\hdots, x_n^{d}\}$ and let $r$ be any monomial of degree $kd$ with $k \geq 2$. We want to see that $r$ can be factorized as product of $k$ monomials in $\Omega_{n,d} = \cM_{n,d} \setminus \{m\}$. If $m$ does not divide $r$, then the result is true. Moreover, if $\supp(r) \subseteq \supp(m)$, then the result follows from (i). 
Therefore, we can assume that $m = x_0^{a_0}\cdots x_{t}^{a_t}$ with $0 < t < n$ and $a_0\cdots a_t \neq 0$, and $r = x_{0}^{b_0} \cdots x_{n}^{b_n}$ with $a_0 \leq b_0, \hdots, a_t \leq b_n$ and there is $j$ with $t< j \leq n$ such that $b_j > 0$. We set $s_1 = m\frac{x_j}{x_0} \in \Omega_{n,d}$. Since $s_1$ divides $r$, we get a monomial $s_2 = r/s_1$ of degree $(k-1)d$. For $k = 2$, the product $r = s_1s_2$ is  as wanted unless $s_2 = m$. In this case, $r = m^2\frac{x_j}{x_0}$, and it is enough to take 
$s_1 = m\frac{x_j}{x_1}$ and $s_2 = r/s_1 \neq m$. Now the result follows by induction on $k$.
\end{proof}

The bound $\lceil (n+2)/2 \rceil$ for the support in Proposition \ref{Propsition: 2-normality} (i) can not be replaced by a smaller number while maintaining 2-normality. Take for example $n=3$ and $d = 5$. Then $\lceil (n+2)/2 \rceil = 3$, and the monomial $x_0^2x_1^2x_2^2x_3^4$ of degree $10$ is not the product of two monomials of degree $5$ in two variables.    
Similarly, if we violate (ii) by taking $\Omega_{n,d} = \cM_{n,d} \setminus \{x_0^d\}$, then the $\kk$-vector space spanned by $\Omega_{n,d}$ is not $2$-normal. 

\begin{remark}\label{Remark: non CM} The question of the arithmetically Cohen-Macaulayness of monomial projections $Y_{n,d}$ is a longstanding problem posed by Gr\"obner in \cite{Groebner}, see also \cite{ThesisLiena, CMM-R, H, HT, Ho, Groebner, S, T, T1}. For $\Omega_{n,d} \neq \cM_{n,d}$ as in Proposition \ref{Propsition: 2-normality}, a straightforward application of \cite[Theorem 2.6]{Goto-Suzuki-Watanabe} and the $2$-normality property shows that $\kk[\Omega_{n,d}]$ is always a non Cohen-Macaulay ring. 
\end{remark}

The proof of Theorem 3.10 relies on the fact that the defining ideal of $\kk[V]$ is generated in degree two and three, when $V \subset R_d$ is a $2$-normal subspace, proved in \cite[Proposition 1.2]{Alzati-Russo}. We provide an alternative proof of this claim, in the toric case.

\begin{proposition}\label{Proposition: generation at most 3}
Let $\Omega_{n,d} \subset \cM_{n,d}$ be a subset of monomials spanning a 2-normal vector space in $R_d$, and let $Y_{n,d}$ be the monomial projection parameterized by $\Omega_{n,d}$. Then, the ideal $\I(Y_{n,d})$ is generated by binomials of degrees at most 3.
\end{proposition}

\begin{proof} Fix $k \geq 4$ and let $w^{\alpha} = w^{\alpha_{+}}-w^{\alpha_{-}} = w_{i_{1}}\cdots w_{i_{k}} - w_{j_{1}}\cdots w_{j_{k}} \in \I(Y_{n,d})_k$ be a non-trivial $k$-binomial.  For each $w_{i_{\ell}}$ (respectively $w_{j_{\ell}}$), let $m_{i_{\ell}} = x_{0}^{a_{\ell 0}}\cdots x_{n}^{a_{\ell n}} \in \Omega_{n,d}$ be its associated monomial (respectively, $m_{j_{l}} = x_{0}^{b_{\ell 0}}\cdots x_{n}^{b_{\ell n}} \in \Omega_{n,d}$), for $\ell = 1,\hdots,k$. 

We consider the monomials $m_{i_1}$ and $m_{j_1}$, and for each $0 \leq s \leq n$, we define:
\[c_{s} = \left\{\begin{array}{lll}
0 & \quad &\text{if} \;\; a_{1s} \geq b_{1s}\\
b_{1s}-a_{1s} & \quad & \text{otherwise.}
\end{array}\right.\]
This gives rise a non zero monomial $\widetilde{m} = x_{0}^{c_{0}}\cdots x_{n}^{c_{n}} \in R$ of degree strictly less than $d$, which divides $m_{i_{2}}\cdots m_{i_{k}}$. Thus, we consider $m' = (m_{i_{2}}\cdots m_{i_{k}})/\tilde{m}$, which is a monomial of degree at least $(k-2)d \geq 2d$. Then, arguing as in the proof of Proposition \ref{Propsition: 2-normality}, we can find a monomial $r \in \Omega_{n,d}$ of degree $d$ dividing $m'$. Thus, $m_{i_{2}}\cdots m_{i_{k}} = r\cdot s$, where $s$ is a monomial of degree $(k-2)d \geq 2d$. By hypothesis, $\langle \Omega_{n,d} \rangle$ is $2$-normal. Since $s$ is a monomial of degree at least $2d$, we can factorize
\[m_{i_{2}}\cdots m_{i_{k}} = (m_{i_{2}}^{(1)} \cdots m_{i_{k-1}}^{(1)})m_{i_{k}}^{(1)},\]
where all $m_{i_{l}}^{(1)} \in \Omega_{n,d}$, for $2 \leq l \leq k$. In particular,
$m_{i_{k}}^{(1)} = r.$ Now define $w^{(2)} \in S$ to be the monomial $\rho^{-1}(m_{i_{(1)}})\rho^{-1}(m_{i_{2}}^{(1)}) \cdots \rho^{-1}(m_{i_{k}}^{(1)})$. By construction, $w^{\alpha_{+}}-w^{(2)} \in \I(Y_{n,d})_k$ is a  trivial $k$-binomial and $\widetilde{m}$ divides $m_{i_{2}}^{(1)}\cdots m_{i_{k-1}}^{(1)}$. Thus, $m_{j_{1}}$ divides $m_{i_{1}}m_{i_{2}}^{(1)}\cdots m_{i_{k-1}}^{(1)}$. Applying the same argument as before, we factorize
\[m_{i_{1}}m_{i_{2}}^{(1)}\cdots m_{i_{k-1}}^{(1)} = m_{i_{1}}^{(2)}m_{i_{2}}^{(2)}\cdots m_{i_{k-1}}^{(2)},\]
where $m_{i_{1}}^{(2)} = m_{j_{1}}$ and all $m_{i_{\ell}}^{(2)} \in \Omega_{n,d}$, for $2 \leq \ell \leq k-1$. We set 
$$w^{(3)} = \rho^{-1}(m_{i_{1}}^{(2)})\cdots \rho^{-1}(m_{i_{k-1}}^{(2)}) \rho^{-1}(m_{i_{k}}^{(1)}).$$
Since 
$$m_{i_{1}}m_{i_{2}}^{(1)}\cdots m_{i_{k-1}}^{(1)}m_{i_{k}}^{(1)} = m_{i_{1}}^{(2)}m_{i_{2}}^{(2)}\cdots m_{i_{k-1}}^{(2)}m_{i_{k}}^{(1)},$$
$w^{(2)}-w^{(3)} \in \I(Y_{n,d})_k$ is a trivial $k$-binomial. Furthermore, since
$m_{i_{1}}^{(2)} = m_{j_{1}}$, also $w^{(3)}-w^{\alpha_{-}} \in \I(Y_{n,d})_k$ is a trivial $k$-binomial. Therefore, $(w^{\alpha_+},w^{(2)},w^{(3)},w^{\alpha_-})$
is an $\I(Y_{n,d})_{k}$-sequence, and by Proposition \ref{Proposition: Criterion generation}, we have $w^{\alpha} \in (\I(Y_{n,d})_{k-1})$. Repeating the same argument, we obtain
\[\cdots \subset (\I(Y_{n,d})_{k}) \subset (\I(Y_{n,d})_{k-1}) \subset \cdots \subset (\I(Y_{n,d})_{3}). \qedhere\]
\end{proof}

\begin{corollary}\label{coro} Let $\Omega_{n,d} \subset \cM_{n,d}$ be a subset of monomials parametrizing a  monomial projection
 $Y_{n,d}$. Assume
\begin{itemize}
\item[(i)]  $\Omega_{n,d} \supseteq \{ m \in \cM_{n,d} \ | \ |\supp(m)| \le \lceil(n+2)/2 \rceil \}$, or
\item[(ii)] $\Omega_{n,d} = \cM_{n,d}\setminus \{m\}$ with $m \in \cM_{n,d} \setminus \{x_0^{d},\hdots, x_{n}^{d}\}$. 
\end{itemize}
Then, the ideal $\I(Y_{n,d})$ is generated by binomials of degree at most 3.
\end{corollary}
\begin{proof}
It immediately follows from Propositions \ref{Propsition: 2-normality} and \ref{Proposition: generation at most 3}.
\end{proof}

Our next question is when degree 3 generators of $\I(Y_{n,d})$ are necessary. Indeed, we see in the example below that quadrics are not always enough to generate $\I(Y_{n,d})$ for all monomial projections covered in Corollary \ref{coro}.

\begin{example} Take $n = 2$, $d = 4$ and $\Omega_{2,4} = \cM_{2,4} \setminus \{x_0^2x_1^2\}$. We have checked computationally, using the software Macaulay2 \cite{Macaulay2}, that the homogeneous ideal $\I(Y_{2,4})$ of the monomial projection $Y_{2,4}$ parameterized by $\Omega_{2,4}$ is minimally generated by $60$ quadrics and $3$ cubics. 
\end{example}

We are now ready to state the main result of this section, which also 
  provides a partial answer to Question 16 in \cite{Conca-DeNegri-Rossi} concerning the {\em pinched Veronese algebra} $\PV (n,d,s)$, i.\,e.\ the algebra generated by all monomials in $n+1$ variables of degree
$d$ that are supported in at most $s$ variables.

\begin{theorem}\label{thm:pv_quadratic}
  Fix  integers $d,n \geq 2$, and let $Y_{n,d}$ be a monomial projection parameterized by $\Omega_{n,d} \supset \{ m \in \cM_{n,d} \ | \ |\supp(m)| \le  \lceil \frac{n+2}{2} \rceil \}$. Then, the algebra  $\A(Y_{n,d})$ is quadratic. In particular, when $s \ge \lceil \frac{n+2}{2} \rceil$, the pinched Veronese algebra $\PV (n,d,s)$ is quadratic.
\end{theorem}
\begin{proof} Let $t= \lceil (n+2)/2 \rceil$. By Proposition \ref{Proposition: Criterion generation} and Corollary \ref{coro},  the ideal $\I(Y_{2,d})$ is generated in degrees two and three, so it remain to show that generators of degree three are not needed. To this end, we take a nontrivial relation
\[
m_1m_2m_3=m_4m_5m_6=x_0^{d_0} x_1^{d_1} \cdots x_n^{d_n}, \ \ m_1, \ldots, m_6 \in \Omega_{n,d}. 
\]
As we can apply the factorization (\ref{eq:2normal_decomp}) from the proof of Proposition \ref{Propsition: 2-normality} pairwise to the monomials,  it is enough to consider the case $\Omega_{n,d}=\{m \in \cM_{n,d} \ | \ |\supp(m)| \le t \}.$
The goal is to prove the existence of two sequences:
\begin{align*}
&m_1m_2m_3 = m_1^{(1)}m_2^{(1)}m_3^{(1)}= \dots = m_1^{(r)}m_2^{(r)}m_3^{(r)} \\
 & m_4m_5m_6 = m_4^{(1)}m_5^{(1)}m_6^{(1)}= \dots = m_4^{(s)}m_5^{(s)}m_6^{(s)},
\end{align*}
where each step is a degree two relation, so one monomial is unchanged, and $m_i^{(r)}=m_j^{(s)}$ for a pair $i \in \{1,2,3\}$, $j \in \{4,5,6\}$.  
We distinguish three cases.

\vskip 1mm
\noindent \underline{Case 1:} $n=2$. Here $t=2$ and, after possibly reordering our $x$-variables, we can assume that $d_0 = \min\{d_0,d_1,d_2\}$ and so $d_0 \le d$.
We claim there are $m_2^{(1)}$ and $m_3^{(1)}$ in $\Omega _{2,d}$ such that $m_2m_3=m_2^{(1)}m_3^{(1)}$ and $x_0 \notin \supp(m_3^{(1)})$. Indeed, say $m_2m_3=x_0^{a_0}x_1^{a_1}x_2^{a_2}$. As $d_0 \le d$, we also have $a_0 \le d$. Then $a_0+a_1 \ge d$ or $a_0+a_2 \ge d$, as both strictly less than $d$ would disagree with $a_0+a_1+a_2=2d$. Say $a_0+a_1 \ge d$. Then take $m_2^{(1)}=x_0^{a_0}x_1^{d-a_0}$ and $m_3^{(1)}=x_1^{d-a_2}x_2^{a_2}$. 
Now we set $m_1^{(1)}=m_1$ and we apply the same procedure to $m_1^{(1)}m_2^{(1)}$. We obtain $m_1^{(2)}$ and $m_2^{(2)}$ such that $\supp(m_1^{(2)}) \subseteq \{x_0, x_1\}$ and $\supp(m_2^{(2)}) \subseteq \{x_1,x_2\}$. Taking $m_3^{(2)}=m_3^{(1)}$, we now have the sequence $m_1m_2m_3=m_1^{(1)}m_2^{(1)}m_3^{(1)}=m_1^{(2)}m_2^{(2)}m_3^{(2)}$, where we have only used degree two relations. As $m_1^{(2)}$ is the only one divisible by $x_0$, we have $m_1^{(2)}=x_0^{d_0}x_1^{d-d_0}$. With exactly the same arguments, we find a sequence $m_4m_5m_6=m_4^{(1)}m_5^{(1)}m_6^{(1)}=m_4^{(2)}m_5^{(2)}m_6^{(2)}$ so that $m_4^{(2)}=x_0^{d_0}x_1^{d-d_0}=m_1^{(2)}$. 

\vskip 1mm
\noindent \underline{Case 2:} $n=3$. We continue in a similar flavour as the previous case. Here $t=3$ and two of the exponents $d_0,d_1,d_2,d_3$ are less than or equal to $d$ as $d_0+d_1+d_2+d_3=3d$. We may assume $d_0, d_1 \le d$. The key is the following operation. We have $m_1m_2=x_0^{a_0}x_1^{a_1}x_2^{a_2}x_3^{a_3}$, where in particular $a_0,a_1 \le d$. Then, we can refactor $m_1m_2=m_1^{(1)}m_2^{(1)}$ so that $m_1^{(1)}, m_2^{(1)} \in \Omega _{3,d}$ and all $x_0$'s are collected in $m_1^{(1)}$, and all $x_1$'s are collected in $m_2^{(1)}$. In other words, $\supp(m_1^{(1)}) \subseteq \{x_0,x_2,x_3\}$ and $\supp(m_2^{(1)}) \subseteq \{x_1,x_2,x_3\}$. Repeating similar arguments, we can achieve a sequence 
\[
m_1m_2m_3 = m_1^{(1)}m_2^{(1)}m_3^{(1)}=m_1^{(2)}m_2^{(2)}m_3^{(2)}=m_1^{(3)}m_2^{(3)}m_3^{(3)}\]
where $m_1^{(3)}=x_0^{d_0}f_1, \; m_2^{(3)}=x_1^{d_1}f_2, \; m_3^{(3)}=f_3$ and $f_1,f_2,f_3$ are monomials in $x_2,x_3$. The same procedure can be applied to $m_4m_5m_6$. We get $m_4^{(3)}m_5^{(3)}m_6^{(3)}=(x_0^{d_0}f_4)(x_1^{d_1}f_5)f_6$ where $f_4,f_5,f_6$ are monomials in $x_2,x_3$.  We are not done yet, as $f_1$ and $f_4$ are not necessarily equal. Say $f_1$ has smaller $x_2$-degree compared to $f_4$. Then one of $f_2$ or $f_3$ is divisible by $x_2$, so we can apply one of the quadratic relations
\[
(x_0^{d_0}f_1)(x_1^{d_1}f_2) = (x_0^{d_0}\frac{x_2}{x_3}f_1)(x_1^{d_1}\frac{x_3}{x_2}f_2) \quad \mbox{or} \quad (x_0^{d_0}f_1)f_3 =(x_0^{d_0}\frac{x_2}{x_3}f_1)\frac{x_3}{x_2}f_3. 
\]
We can repeat this argument until $f_1$ has been replaced by $f_4$.

\vskip 2mm
\noindent \underline{Case 3:} $n>3$.
 In the following remark, we collect an operation that will be applied several times in the proof of this third case.

\begin{remark} \label{lem:operation}
Take $m, \widetilde{m} \in \Omega_{n,d}$ both divisible by some variable $x_i$. Assume further that $x_j | m$ for some $j \ne i$ and let
 \[ m^{(1)} = \frac{x_i}{x_j}m, \quad \widetilde{m}^{(1)}= \frac{x_j}{x_i}\widetilde{m}.  \]
Then $m^{(1)}\widetilde{m}^{(1)} = m\widetilde{m}$ and if $x_j | \widetilde{m}$ or $|\supp(\widetilde{m})|<t$ we have $m^{(1)}, \widetilde{m}^{(1)} \in \Omega_{n,d}$. 
\end{remark}
We will use the notation $(a_{i,0}, a_{i,1}, \ldots, a_{i,n})$ for the exponent vector of $m_i$, $i=1, \ldots, 6$. Let $i,j$ be a pair $i \in \{1,2,3\}$, $j\in \{4,5,6\}$ such that $|\supp(m_i) \cap \supp(m_j)|$ is maximal among the choices of $i,j$. We can assume $i=1$ and $j=4$. 
We may index our variables so that $\supp(m_1) \cap \supp(m_4) = \{x_0,x_1, \ldots, x_\beta\}$, and $d_0 \le d_1 \le \dots \le d_\beta$. In the first part of the proof, we will prove the existence of a sequence such that $m_1^{(r)}$ and $m_4^{(s)}$ both are divisible by $x_0^{d_0} \cdots x_\beta^{d_\beta}$, if $d_0 + \dots + d_\beta \le d$. If $d_0 + \dots + d_\beta > d$, we will instead find a sequence with 
\[
m_1^{(r)}=m_4^{(s)}=\Big(\prod_{0\le i<k} x_i^{d_i}\Big) x_k^a, \quad \mbox{for a} \ k \le \beta.
\]
As a first step, we aim to obtain a sequence such that $m_1^{(*)}$ is divisible by $x_0^c$ where $c=\min(d_0,d)$.  We assume $a_{1,0}<\min(d_0,d)$, as otherwise we are already done with this first step. Moreover, applying the factorization (\ref{eq:2normal_decomp}) to the monomial $m_2m_3$, we can assume that $|\supp(m_2) \cap \supp(m_3)|\leq 1$. If we are in the hypothesis of Remark \ref{lem:operation} with respect to $m_1$ and $m_2$, or $m_1$ and $m_3$, then we can get the initial step $m_1m_2m_3 = m_1^{(1)}m_2^{(1)}m_3^{(1)}$ in our sequence and $a_{1,0}^{(1)} = a_{1,0}+1$. Else, we are in one of the following cases:
\begin{enumerate}[(i)]
 \item $\supp(m_1) \cap \supp(m_2) = \{x_0\}$, $|\supp(m_2)| = t$ and $a_{3,0}=0$,
 \item $\supp(m_1) \cap \supp(m_2) = \supp(m_1) \cap \supp(m_3) =\{x_0\}$ and $|\supp(m_2)|=|\supp(m_3)| = t$. 
 \end{enumerate} 
Since we have the hypothesis $|\supp(m_2)\cap \supp(m_3)| \leq 1$, we have $|\supp(m_2m_3)| = n+1$, so (ii) reduces to $m_1=x_0^d$  and  $m_1m_2m_3$ is already on the desired form. On the other hand, we can assume in (i) that $\supp(m_2) = \{x_0,x_t,\hdots,x_n\}$ and $\supp(m_3) \subseteq \{x_1,\hdots,x_{t-1},x_{t}\}$.  We first observe the following. If $|\supp(m_{3})| < |\supp(m_2)| = t$ and $x_{t} \in \supp(m_3)$, then we apply Remark \ref{lem:operation} to $m_2$ and $m_3$ with $(i,j) = (t,0)$. Now, we are in position to apply Remark \ref{lem:operation} to $m_1$ and $m_3^{(1)}$. If $|\supp(m_{3})| < |\supp(m_2)| = t$ and $x_{t} \notin \supp(m_3)$, there is some pair $i\ne j$ such that $a_{3,j} > a_{2,i}>0$. Thus, we can do the operation
\begin{equation}\label{eq:op_additional}
m_1^{(1)} = m_1, \quad m_2^{(1)} = \frac{x_j^{a_{2,i}}}{x_i^{a_{2,i}}} m_2, \quad m_3^{(1)} = \frac{x_i^{a_{2,i}}}{x_j^{a_{2,i}}} m_3.
\end{equation}
Either we get in position to apply Remark \ref{lem:operation} to increase $a_{1,0}$, or we get $|\supp(m_{3}^{(1)})| < t$ and $|\supp(m_{2}^{(1)}) \cap \supp(m_{3}^{(1)})| = 1$, and we can argue as before, or we get in (i) with $|\supp(m_{3})| = t$. Thus, it remains to consider (i) with 
 \[
 m_2=x_0^{a_{2,0}}x_t^{a_{2,t}}x_{t+1}^{d_{t+1}} \cdots x_n^{d_n}, \quad    m_3=x_1^{a_{3,1}} \cdots x_t^{a_{3,t}}
 \]
and $|\supp(m_2)| = |\supp(m_3)| = t$. 
It's straightforward to check that the assumption $|\supp(m_2)|=|\supp(m_3)|=t$ implies that there is a pair $i\in\{1,2,3\}$, $j\in \{4,5,6\}$ such that $m_i$, $m_j$ share at least two variables. Hence, we can assume $\beta >0$. So, we have $m_1m_3 = x_0^{a_{1,0}}x_1^{d_1}\cdots x_{t-1}^{d_{t-1}}x_{t}^{a_{3,t}}$ and by hypothesis $d_0 \leq d_1 \leq \cdots \leq d_{\beta}$.  We distinguish the following cases. 
 
\noindent \underline{$a_{1,0}+d_1 \ge d$ and $d_0>d$.} Then, we have $d_t + \dots + d_n<d$, as otherwise the total degree of $m_1m_2m_3$ would exceed $3d$. This allows us to define
\[
m_1^{(1)}=m_1, \quad m_2^{(1)} = x_0^{a_{2,0}-a_{3,t}}x_t^{d_t} \cdots x_n^{d_n}, \quad m_3^{(1)} = x_0^{a_{3,t}}x_1^{a_{3,1}} \cdots x_{t-1}^{a_{3,t-1}}.
\]
Now we can apply Remark \ref{lem:operation} to $m_1^{(1)}$ and $m_3^{(1)}$. 

\noindent \underline{$a_{1,0}+d_1 \ge d$ and $d_0 \le d$.} Let
 \[
 m_1^{(1)}=x_0^{a_{1,0}}x_1^{d-a_{1,0}} , \quad m_2^{(1)}=m_2, \quad m_3^{(1)} = x_{1}^{d_1-d+a_{1,0}}x_{2}^{d_{2}} \cdots x_{t-1}^{d_{t-1}} x_t^{d_{3,t}},
 \]
 and, then, 
 \[
 m_1^{(2)}=\frac{x_0^{a_{2,0}}}{x_1^{a_{2,0}}}m_1^{(1)}=  x_0^{d_0}x_1^{d-d_0}, \quad m_2^{(2)}=\frac{x_1^{a_{2,0}}}{x_0^{a_{2,0}}}m_2^{(1)}=x_1^{a_{2,0}}x_t^{a_{2,t}}x_{t+1}^{d_{t+1}}, \quad m_3^{(2)}=m_3^{(1)}. 
 \]
 Now $m_1^{(2)}$ is on the desired form. 
 
\noindent  \underline{$a_{1,0}+d_1<d$.} In this subcase, we can define
 \[
 m_1^{(1)} = x_0^{a_{1,0}}x_1^{d_1} \cdots x_{k-1}^{d_{k-1}}x_{k}^{a_{1,k}^{(1)}}, \quad m_2^{(1)}=m_2, \quad m_3^{(1)} = x_{k}^{a_{3,k}^{(1)}}x_{k+1}^{d_{k+1}} \cdots x_{t-1}^{d_{t-1}} x_t^{a_{3,t}}, 
 \]
 for some $1<k<t$. Next we apply (\ref{eq:op_additional}) once more. Now none of $m_2^{(1)}$ and $m_3^{(1)}$ are divisible by $x_1$. So, we can not possibly end up in the situation where $m_3^{(*)}$ gets maximal support, as this would mean $\supp(m_2^{(1)}m_3^{(1)})$ containing all variables. Hence, we will at last be able to apply Remark \ref{lem:operation}.

Following the steps described above, we will eventually obtain a $m_1^{(*)}$ divisible by $x_0^c$, $c=\min(d_0,d)$. If $c<d$, we obtain a sequence $m_{1}m_{2}m_{3} = \cdots = m_{1}^{(*)}m_2^{(*)}m_{3}^{(*)}$ where $\supp(m_{2}^{(*)}m_{3}^{(*)}) \subset \{x_1,\hdots, x_n\}$. We repeat the same process to increase now $a_{1,1}^{(*)}$. Moreover, we can perform it without involving the variable $x_0$ and neither cases (i) with $|\supp(3)| = t$ or (ii) will occur. Continuing in this way, we obtain a sequence $m_{1}m_2m_3 = \cdots = m_{1}^{(r)}m_2^{(r)}m_3^{(r)}$ where $m_{1}^{(r)} = x_{0}^{d_{0}}\cdots x_{i}^{d_{i}}x_{i+1}^{a_{1,i+1}^{(r)}}$ and $\supp(m_{2}^{(r)}m_{3}^{(r)}) \subset \{x_{i+1},\hdots,x_n\}$. We apply the steps described above to $m_4m_5m_6$ as well. If $d_0 + d_1 + \dots +d_\beta \geq d$, this means we have obtained $m_1^{(r)} = m_4^{(s)}$, and we are done. Else, we now have $m_1^{(r)}$ and $m_4^{(s)}$ divisible by $x_0^{d_0} \cdots x_\beta^{d_\beta}$ but $m_1^{(r)} \ne m_4^{(s)}$. Let's reset the notation here so that $m_i:=m_i^{(r)}$ for $i=1,2,3$ and $m_j:=m_j^{(s)}$ for $j=4,5,6$. We have
\begin{equation}\label{eq:part1_done}
m_1=x_0^{d_0} \cdots x_\beta^{d_\beta} \widetilde{m}_1, \quad m_4=x_0^{d_0} \cdots x_\beta^{d_\beta} \widetilde{m}_4
\end{equation}
where $\supp(\widetilde m_1) \cap \supp(\widetilde m_4) = \emptyset$, $|\supp(m_{1}) \cap \supp(m_{2}m_{3})| \leq 1$ and $|\supp(m_{4}) \cap \supp(m_5m_6)| \leq 1$. If $|\supp(m_2m_3)| < n$, $m_2m_3$ is a monomial in at most $n-1$ variables, we can assume they are supported in at most $t-1$ variables each. At this step we need to split into two cases. 

\noindent \underline{Case A:} $|\supp(m_1^{(1)})|=|\supp(m_4^{(1)})|=t$. Here
\[
|\supp(m_2^{(1)}m_3^{(1)})| \le n+1-t+1 \le t.
\]
Hence, we can factor $m_2^{(1)}m_3^{(1)}$ in any way we like without exceeding the maximal allowed support. As $\widetilde m_4$ divides $m_2^{(1)}m_3^{(1)}$, we choose $m_2^{(1)}m_3^{(1)}=m_2^{(2)}m_3^{(2)}$ such that $m_2^{(2)}=\widetilde m_4 \widetilde m_2^{(2)}$. In the same way, we factor  $m_5^{(1)}m_6^{(1)}=m_5^{(2)}m_6^{(2)}$ such that $m_5^{(2)}=\widetilde m_1 \widetilde m_5^{(2)}$. But then $\widetilde m_2^{(2)}m_3^{(2)} =\widetilde m_5^{(2)}m_6^{(2)}$, and we can refactor $m_2^{(2)}m_3^{(2)}$ once more so that $m_3^{(3)} = m_6^{(2)}$. 

\noindent \underline{Case B:} $|\supp(m_1^{(1)})|<t$ (or $|\supp(m_4^{(1)})|<t$). Take some $x_i$ in $\supp(\widetilde m_1)$  and $x_j$ in $\supp(\widetilde m_4)$. Then $x_j$ divides $m_2^{(1)}$ or $m_3^{(1)}$, let's say $m_2^{(1)}$. As both $m_1^{(1)}$ and $m_2^{(1)}$ have support smaller that $t$, we can do the operation
\[
m_1^{(2)} = \frac{x_j}{x_i}m_1^{(1)}, \quad m_2^{(2)} = \frac{x_i}{x_j}m_2^{(1)}.
\]
Now $|\supp(m_1^{(2)}) \cap \supp(m_4^{(1)})|>|\supp(m_1^{(1)}) \cap \supp(m_4^{(1)})|$.
So, we go back to the first part of the proof, which will result in (\ref{eq:part1_done}) with a greater value of $\beta$.  

Finally, when $|\supp(m_2m_3)| = n$,  we have $\beta = 0$ and we can assume that $\supp(m_1) = \{x_0,x_1\}$. When $n \geq 4$, the allowed support is $t \geq 3$. So, if $x_{i} \in \supp(m_4)$ with $i \neq 0$, then we can go back to the first part of the proof to increase $a_{1,i}$ as if $x_i \in \supp(m_1)$. Either we obtain the desired sequence in the process or we end with $\beta > 0$ and we finish the proof.
\end{proof}

For $s<\lceil \frac{n+2}{2} \rceil$, the pinched Veronese $\PV(n,d,s)$ is in general not quadratic. Take for instance $n=3$ and $d=5$. Then, $\lceil \frac{n+2}{2} \rceil=3$, and the ideal $\I(Y_{3,5})$ defining the algebra $\PV(3,5,2)$ is minimally generated by $168$ binomials of degree $2$ and $12$ binomials of degree $3$.

\subsection{ Monomial projections parameterized by invariants of finite abelian groups.}\label{subsec:monomials_group_inv} In this subsection, we shift focus to our second family of monomial projections. 
We fix an integer 
$n \geq 2$ and $G \subset \GL(n+1,\kk)$ a finite abelian group of order $d$. We will assume that $\kk$ is an algebraically closed field and $d$ is  not divisible by the characteristic of  $\kk$. Therefore, we can assume that $G$ acts diagonally on $R$. We write $G = \Gamma_1 \oplus \cdots \oplus \Gamma_{s}$,  where each $\Gamma_i \subset \GL(n+1,\kk)$ is a cyclic group of order $d_i$ with $d = d_1\cdots d_s$ and $d_i|d_{i+1}$.  We present $\Gamma_i$ as the group generated by a diagonal matrix
\[M_{d_i;\alpha_{i0}, \hdots,\alpha_{in}} := \diag(e_i^{\alpha_{i0}}, \hdots,
e_i^{\alpha_{in}}) = \left(\begin{array}{lllllll}
e_i^{\alpha_{i0}} & 0 & \cdots & 0\\
0                & e_i^{\alpha_{i1}} & \cdots & 0\\
\vdots & \vdots & \ddots & \vdots\\
0 & 0 & \cdots & e_i^{\alpha_{in}}\\
\end{array}\right),
\]
where $e_i$ is a $d_i$-th primitive root of $1 \in \kk$ and $\GCD(\alpha_{i0}, \hdots,\alpha_{in} ,d_i)=1$. 

The {\em first cyclic extension} of $G$ is the finite abelian group $\bG \subset \GL(n+1,\kk)$ generated by $G$ and the diagonal matrix
\[M_{d;1, \hdots,1} := \diag(e,\hdots,e) = \left(\begin{array}{lllllll}
e & 0 & \cdots & 0\\
0                & e & \cdots & 0\\
\vdots & \vdots & \ddots & \vdots\\
0 & 0 & \cdots & e\\
\end{array}\right),
\]
where $e$ is a $d$-th primitive root of $1 \in \kk$.
By \cite[Theorem 2.2.11]{ThesisLiena}, the set $\cB_1$ of monomial invariants of $G$ of degree $d$ minimally generates the ring $R^{\bG}$ of invariants of $\bG$. Let $\mu_d=|\cB_1|$. 
Analogously, we define the {\em $t$-th cyclic extension} of $G$ as $\bG^t = \langle G,M_{td;1,\hdots,1} \rangle \subset \GL(n+1,\kk)$. As a consequence, for any $t \geq 1$ the algebra $R^{\bG^t}$ is minimally generated by the set of monomial invariants of $G$ of degree $td$. 

\begin{notation}\label{DefG-variety} Let $G \subset \GL(n+1,\kk)$ be a finite abelian group of order $d$. By $X_{n,d}^{G} \subset \PP^{\mu_{n,d}-1}$ we denote the monomial projection of the Veronese variety $X_{n,d} \subset \PP^{N_{n,d}-1}$ parameterized by $\cB_1$. 
\end{notation} 

\begin{remark}\label{Remark: CM} \rm The homogeneous coordinate ring $\A(X_{n,d}^G)$ of $X_{n,d}^G$ is isomorphic to $R^{\bG}$. So, $\A(X_{n,d}^G)$ is always a Cohen-Macaulay algebra (\cite[Proposition 13]{Hochster-Eagon}). 
\end{remark}

Let us see an example. 
\begin{example}\label{ex:group_inv_alg}
Take $n=3, d = 4$ and $G = \langle M_{4;0,1,2,3} \rangle \subset \GL(4,\kk)$ a cyclic group of order $3$. Then, $G$ acts on $R =\kk[x_0,x_1,x_2,x_3]$ by $(x_0,x_1,x_2,x_3) \mapsto (x_0,ex_1,e^2x_2,e^3x_3)$, where $e$ is a $d$-th root of unity. Thus, a monomial is in $R^{\bG}$ if and only if its exponent vector $(a_0,a_1,a_2,a_3) \in \ZZ_{\geq 0}^4$ is a solution of the linear system of congruences
\[\left\{\begin{array}{llllllllll}
y_0 & + & y_1 & + & y_2 & + &y_3 & \equiv & 0 \mod d\\
    &   & y_1 & + &2y_2 & + &3y_3& \equiv & 0 \mod d.
\end{array}\right.\]
The monomials of degree $4$ and $8$ in $R^{\bG}$ are respectively:
\[\begin{array}{lll}
\cB_1 &=& \{x_0^4, x_1^4, x_0x_1^2x_2, x_0^2x_2^2, x_0^2x_1x_3, x_2^4, 
x_1x_2^2x_3, x_1^2x_3^2, x_0x_2x_3^2, x_3^4\}.\\
\cB_2 &=& \{x_0^8, x_1^8, x_0x_1^6x_2, x_0^2x_1^4x_2^2, x_0^3x_1^2x_2^3, x_0^4x_2^4, 
x_3x_0^2x_1^5, x_3x_0^3x_1^3x_2, x_3x_0^4x_1x_2^2, x_3^2x_0^4x_1^2, x_3^2x_0^5x_2, x_2^8,\\ & & x_3x_1x_2^6, x_3^2x_1^2x_2^4, x_3^2x_0x_2^5, x_3^3x_1^3x_2^2, x_3^3x_0x_1x_2 3, 
x_3^4x_1^4, x_3^4 x_0x_1^2x_2, x_3^4x_0^2x_2^2, x_3^5x_0^2x_1, x_3^8\}.
\end{array}\]
We have that $R^{\bG} = \kk[\cB_1]$ and $R^{\bG^2} = \kk[\cB_2]$. The sets $\cB_1$ and $\cB_2$ parameterize monomial projections $X_{3,4}^G$ and $X_{3,8}^G$ of the Veronese threefolds $X_{3,4} \subset \PP^{34}$ and $X_{3,8} \subset \PP^{164}$. 
\end{example}

In contrast to the monomial projections considered in Section \ref{subsec:monomials_group_inv}, the $\kk$-vector spaces spanned by $\cB_1$ are rarely 2-normal. Take $G = \langle M_{d,\alpha_0,\hdots, \alpha_n} \rangle \subset \GL(n+1,\kk)$ a cyclic group of order $d$, and assume $0 \neq \alpha_i < \alpha_j$ for some $0 \leq i,j \leq n$. For any integer $t \geq 1$, the monomial $x_i^{td} \in R^{\bG}$. However, the monomial $x_i^{td-1}x_j$ is not an invariant of $G$. Indeed, $x_i^{td-1}x_j \in R^{\bG}$ if and only if $td\alpha_i+ \alpha_j - \alpha_i \equiv 0 \mod d$, which happens if and only if $\alpha_j = \alpha_i$. Thus, $\dim_{\kk}\kk[\cB_1]_t < \dim_{\kk} R_{td}$ for all $t \geq 1$, so $\cB_1$ is non $2$-normal. 
Still, we have the property that the homogeneous ideal  $\I(X_{n,d}^G)$ of the variety $X_{n,d}^{G}$ is generated by binomials of degree at most $3$, see \cite[Theorem 3.2.6]{ThesisLiena}.  Hence, we are led to consider the following problem.

\begin{problem} \label{pblm2} \rm
For which finite abelian groups $G \subset \GL(n+1,\kk)$ of order $d$ are the algebras $\A(X_{n,d}^G)$ quadratic?
\end{problem}

Indeed, there are finite abelian groups $G$ for which the ideals $\I(X_{n,d}^G)$ are minimally generated by binomials of degrees $2$ and $3$. But we also have families of varieties $X_{n,d}^G$ whose homogeneous ideals are minimally generated by binomials of degree $2$. To describe these families, we need to fix some notation. 

\begin{notation}From now on, given integers
$d \geq 2$ and $0 < \alpha_1 < \alpha_2 < d$, we set $\alpha_1' = \frac{\alpha_1}{\GCD(\alpha_1,d)}$, $d' = \frac{d}{\GCD(\alpha_1,d)}$ and let $\mu, \lambda$ be the unique determined integers with $0 < \lambda \leq d'$ such that $\alpha_2 = \lambda\alpha_1'+\mu d'$.
\end{notation}

In addition, the strictness of the inequality 
\begin{equation}\label{eq:ineq_gcd}
    \GCD(\alpha_1,d)\cdot\GCD(\lambda,d')\cdot\GCD(\lambda - \GCD(\alpha_1,d),d') \ge 1
\end{equation}
plays an important role. 

\begin{proposition}\label{Proposition: degree binomial GT-surfaces} Let $G = \langle M_{d;0,\alpha_1,\alpha_2} \rangle \subset\GL(3,\kk)$ be cyclic group of order $d \geq 2$ with $0 < \alpha_1 < \alpha_2 < d$. Then, $\A(X_{2,d}^G)$ is a quadratic algebra if and only if the inequality (\ref{eq:ineq_gcd}) is strict. 
\end{proposition}
\begin{proof} See \cite[Corollary 3.1.24 and Proposition 3.2.8]{ThesisLiena}.
\end{proof}

\begin{proposition}\label{Proposition: degree binomial d prime} Let $G = \langle M_{d;0,\rho _1, \hdots, \rho_n} \rangle \subset \GL(n+1,\kk)$ be a cyclic group $G = \langle M_{d;0,\rho _1, \hdots, \rho_n} \rangle \subset \GL(n+1,\kk)$ of order $d$. If there are entries $\rho_i < \rho_j < \rho_k$ such that we have equality in
(\ref{eq:ineq_gcd}) with $\alpha_1=\rho_j-\rho_i$ and $\alpha_2=\rho_k-\rho_i$, then $\A(X_{n,d}^{G})$ is not quadratic.
\end{proposition}
\begin{proof} It follows directly from Proposition \ref{Proposition: degree binomial GT-surfaces} taking into account that $\I(X_{n,d}^{G})$ contains the binomial generators of the surface $X_{2,d}^{G_1}$, where $G_1 = \langle M_{d;0,\alpha_j-\alpha_i,\alpha_k-\alpha_i} \rangle \subset \GL(3,\kk)$. Indeed, by construction, if a binomial of degree $3$ in $\I(X_{2,d}^{G_1})$ does not admit a $\I(X_{2,d}^{G_1})_3$-sequence, then it does not admit an $\I(X_{n,d}^G)_3$-sequence. So by Proposition \ref{Proposition: Criterion generation}, it cannot be expressed as a linear combination of binomials of degree $2$ in $\I(X_{n,d}^{G})$.
\end{proof}
In the next section, applying Lemma \ref{lemma:GB_extension}, we will get examples of varieties $X_{n,d
}^G$ of any dimension $n\ge 2$ such that the associated algebra $\A(X_{n,d}^G)$ is quadratic. We finish this section with an example studied in \cite{CM-R}, for which the associated algebra is quadratic.  

\begin{proposition}\label{Proposition: degree binomial GT-threefold}
 Let $G = \langle M_{d;0,1,2,3} \rangle \subset \GL(4,\kk)$ be a cyclic group  of order $d \geq 4$. Then, $\A(X_{n,d}^G)$ is a quadratic algebra if and only if $d$ is even.
\end{proposition}
\begin{proof}
See \cite[Corollary 5.7]{CM-R}.
\end{proof}

 Motivated by the results of this section, we  will investigate the Koszulness and the existence of a quadratic  Gr\"obner basis for the ideals $\I(X_{n,d}^G)$. 

\section{Koszulness of monomial projections of Veronese varieties}
\label{Section: Koszulness}
This section is entirely devoted to studying the Koszulness of the homogeneous coordinate ring $\kk[\Omega_{n,d}]$  of a monomial projection of a Veronese variety, as well as whether $\kk[\Omega_{n,d}]$ is G-quadratic.

An application of \cite[Corollary 6.10(2)]{CHTV} states that the $d$-th graded component of a monomial complete intersection ideal
\[
(x_0^{\lambda+1}, \ldots, x_n^{\lambda+1})_{d} = \{ x_0^{a_0} \cdots x_n^{a_n} \in \cM_{n,d} \ | \ a_i>\lambda \ \text{for some} \ i \}
\]
generates a Koszul algebra if $d- \lambda > \frac{(\lambda+1) n}{n+1}$. This provides a proof for Koszulness of some quadratic algebras from Theorem \ref{thm:pv_quadratic}.

\begin{proposition}\label{prop:koszul_ex}
Take positive integers $d, n, \lambda$ and let 
\[
\Omega_{n,d}=\{x_0^{a_0} \cdots x_n^{a_n}\in \cM_{n,d} \ | \ a_i > \lambda \ \text{for some} \ i\}. 
\]
Let $s$ be the greatest integer such that $d >s \lambda$. Then $\Omega_{n,d}$ contains all $m \in \cM_{n,d}$ with $|\supp(m)| \le s$, and $\kk[\Omega_{n,d}]$ is Koszul if 
$s > \frac{(\lambda+1)n}{\lambda(n+1)}.$
\end{proposition}
\begin{proof}
The two inequalities $d >s \lambda$ and $s > \frac{(\lambda+1)n}{\lambda(n+1)}$ imply 
\[
d-\lambda > \lambda(s-1) \ge \frac{(\lambda+1) n}{n+1}. \qedhere
\]
\end{proof}

We single out two families of Koszul algebras from Proposition \ref{prop:koszul_ex}, obtained by removing full support monomials from $\cM_{n,d}$.

\begin{proposition}\label{coro:koszul_ex}
The following two sets of monomials generate Koszul algebras. 
\begin{enumerate}[(i)]
\item $\cM_{n,d} \setminus \{(x_0 \cdots x_n)^\lambda\}$ with $d=\lambda(n+1)$.
\item $\cM_{n,d} \setminus \{x_0^{\lambda-1} (x_1 \cdots x_n)^\lambda, \ldots,x_n^{\lambda-1} (x_0 \cdots x_{n-1})^{\lambda} \}$ with $d=\lambda (n+1) -1$.
\end{enumerate}
\end{proposition}
\begin{proof} Applying Proposition \ref{prop:koszul_ex}, we obtain the result except in (i) when $n = 2$ and $d = 3$. In this case, we have $\kk[\cM_{2,3}\setminus \{x_0x_1x_2\}]$,  which is isomorphic to the homogeneous coordinate ring of the pinched Veronese $\PV(2,3,2)$ and, by \cite[Theorem 3.1]{C}, it is a Koszul algebra. 
\end{proof}

It is not known whether all algebras in Theorem \ref{thm:pv_quadratic} are Koszul. We highlight the first open case, which appears in $n=2$, $d=4$. 

\begin{question} \label{Q1}  Let $Y_{2,4}$ be
the monomial projection  parameterized by $\cM_{2,4} \setminus \{x_0x_1x_2^2\}$. Is $\A(Y_{2,4})$ a Koszul algebra?
\end{question}

\begin{remark}
We have not been able to detect a term order for which the algebra $\A(Y_{2,4})$ in Question \ref{Q1} is G-quadratic. However, this does not exclude the possibility of $\A(Y_{2,4})$ being Koszul. For example, the pinched Veronese $\PV(2,3,2)$ is proved to be Koszul in \cite{C, Caviglia-Conca}, but as remarked in \cite{Conca-DeNegri-Rossi}, there is no term order for which it is G-quadratic in the given coordinates. 
\end{remark}

In the remaining part of this section, we deal with our second family of monomial projections. We fix $\kk$ an algebraically closed field of characteristic zero. Recall that $X_{n,d}^G$ denotes the variety parameterized by the set $\cB_1$ of monomial invariants of degree $d$ of a finite abelian group $G \subset \GL(n+1,\kk)$ of order $d$.  
We start showing that for any surface $X_{2,d}^G$, with $G \subset \GL(3,\kk)$, the algebra $A(X_{2,d}^G)$ is Koszul if and only if it is quadratic. We will use the following result.

\begin{proposition}\label{Proposition: Katthan-Yanagawa} Let $A$ be a standard Cohen-Macaulay domain over a field of characteristic zero with short $h$-vector $(h_0, h_1, h_2 \neq 0)$. If $h_2 < h_1$, then $A$ is Koszul.
\end{proposition}
\begin{proof}
See \cite[Theorem 5.2(1)]{Conca-Iyengar-Nguyen-Romer}.
\end{proof}

By \cite[Proposition 3.1.2]{ThesisLiena}, we have that the Hilbert series of $\A(X_{n,d}^{G})$ is of the form
\[\frac{h_nz^n + h_{n-1}z^{n-1} + \hdots + h_1z + 1}{(1-z)^{n+1}},\]
where $h_i$ coincides with the number of monomials $m = x_0^{a_0}\cdots x_n^{a_n} \in R^{\bG}$ of degree $id$ satisfying $a_0 < d, \hdots, a_n < d$. 
In particular, $h_1 = \codim(X_d) = \mu_d-(n+1)$ and $h_n$ is the number monomials $m$ in $\cB_1$ with $|\supp(m)| = n+1$.
Since $A(X_{n,d}^G)$ is a Cohen-Macaulay ring, the Castelnuovo-Mumford regularity $\reg(A(X_{n,d}^G))$ is the degree of the $h$-polynomial, $\sum_ih_iz^i$, plus one.
Moreover, by  (\cite[Theorem 3.3.5]{ThesisLiena}),  we have 
\[n \leq \reg(\A(X_{n,d}^{G})) \leq n+1,\]
with equality $\reg(\A(X_{n,d}^G)) = n+1$ if and only if there is at least one monomial $m \in R^{\bG}$ with $|\supp(m)| = n+1$, i.\,e.\ $h_n>0$. 

In the case of the surfaces $X_{2,d}^G$, we have $\reg(\A(X_{2,d}^{G})) \leq 3$ and a short $h$-vector $(1, h_1, h_2)$, where $h_2$ could be zero. From the description of the $h_i$ in terms of $\cB_1$, we always have $h_2 \leq h_1$. We distinguish three cases depending on the shape of $(1, h_1, h_2)$.

\begin{itemize}
\item[(A)] If $h_2=0$, then $\reg(\A(X_{2,d}^{G})) = 2$. This implies that $\I(X_{2,d}^{G})$ is generated by binomials of degree $2$ and $\A(X_{2,d}^{G})$ has a linear resolution. Hence, $\A(X_{2,d}^{G})$ is Koszul.

\item[(B)] If $h_2=1$, then $\reg(\A(X_{2,d}^{G})) = 3$ and $\A(X_{2,d}^{G})$ is a Gorenstein ring. From the symmetry of a minimal free resolution of $\A(X_{2,d}^G)$, it follows that $\A(X_{2,d}^{G})$ is generated by binomials of degree $2$ if and only if $X_{2,d}^G$ is not a cubic surface in $\PP^3$. 

\item[(C)] Otherwise $\reg(\A(X_{2,d}^{G})) = 3$ and $\A(X^G_{2,d})$ is a ring with $h_2 > 1$.
\end{itemize}

\begin{theorem}\label{thm:Koszul_surf} Let $G \subset \GL(3,\kk)$ be a finite abelian group of order $d$. The following are equivalent.
\begin{itemize}
    \item [(i)] $\A(X_{2,d}^G)$ is Koszul.
    \item [(ii)] $\cB_1$ contains at least one monomial $m$ with $|\supp(m)| = 2$.
    \item [(iii)] $\A(X_{2,d}^{G})$ is quadratic.
\end{itemize}
\end{theorem}

\begin{proof} For surfaces $X_{2,d}^G$ of type (A) the result is true. Thus, we can assume that $\reg(\A(X_{2,d}^G)) = 3$ and, so, $\A(X_{2,d}^G)$ has  $h$-vector $(1,h_1,h_2\neq 0)$. First, we prove that if $h_2 = h_1$, then $\I(X_{2,d}^{G})$ cannot be only generated  by binomials of degree $2$. Indeed, if $h_2 = h_1$, then any monomial $m \in \cB_1 \setminus \{x_0^{d},x_1^{d},x_2^{d}\}$ has $|\supp(m)| = 3$. Thus, we have at least one non-trivial relation $x_0^dx_1^dx_2^d = m_1m_2m_3$ with $|\supp(m_i)| = 3$. We consider the $3$-binomial $w^{\alpha} = \rho^{-1}(x_0^d)\rho^{-1}(x_1^d)\rho^{-1}(x_2^d) - \rho^{-1}(m_1)\rho^{-1}(m_2)\rho^{-1}(m_3)$. It is straightforward that $w^{\alpha}$ does not admit an $\I(X_{2,d}^{G})_3$-sequence. By Proposition \ref{Proposition: Criterion generation}, $\I(X_{2,d}^{G})$ cannot be only generated by quadrics. 

Thus, we obtain that $\A(X_{2,d}^{G})$ is Koszul if and only if $h_2 < h_1$; if and only if $\I(X_{2,d}^{G})$ is generated by binomials of degree $2$. Since $h_1 = \codim(\A(X_{2,d}^{G})) = |\cB_1|-3$, and $h_1-h_2$ is the number of monomials $m \in \cB_1 \setminus \{x_0^d,x_1^d,x_2^d\}$ with $|\supp(m)| = 2$, the result follows. \end{proof}

As an application of the above result, we get the following.

\begin{proposition}\label{prop:Koszul_surf} If $G = \Gamma_1 \oplus \cdots \oplus \Gamma_s \subset \GL(3,\kk)$ is a finite abelian non cyclic group of order $d = d_1\cdots d_s$, then the algebra $\A(X_{2,d}^{G})$ is Koszul.
\end{proposition}
\begin{proof}
The monomial $x_0^{d_s}x_1^{d-d_s}$ is an invariant of $G$, so $\A(X_{2,d}^{G})$ is Koszul by Theorem \ref{thm:Koszul_surf}.
\end{proof}

Given a finite cyclic group $G = \langle M_{d;\alpha_0,\alpha_1,\alpha_2} \rangle \subset \GL(3,\kk)$, for the purpose of studying the Koszulness of $\A(X_{2,d}^{G})$, we can assume that $\alpha_0 = 0$ and $\alpha_1 \leq \alpha_2$.

\begin{proposition}\label{prop:G_koszul}
Let $G = \langle M_{d;0,\alpha_1,\alpha_2} \rangle \subset \GL(3,\kk)$ be a finite cyclic group of order $d$. The algebra $\A(X_{2,d}^{G})$ is quadratic, and hence Koszul,  precisely when
\begin{itemize}
\item[(i)] $\alpha_1 = \alpha_2$, or
\item[(ii)] $\alpha_1 < \alpha_2$ and the inequality (\ref{eq:ineq_gcd}) is strict. 
\end{itemize}
\end{proposition}
\begin{proof}
(i) It is enough to consider the surface $X_{2,d}^{G}$ where $G = \langle M_{d;0,1,1} \rangle \subset \GL(3,\kk)$. This is of type (A), so $\A(X_{2,d}^{G})$ is Koszul.

\vspace{0.15cm}
\noindent (ii) It follows from Proposition \ref{Proposition: degree binomial GT-surfaces} and Theorem \ref{thm:Koszul_surf}.
\end{proof}

The fact that $\A(X_{2,d}^G)$ being Koszul is equivalent to $\I(X_{2,d}^G)$ being generated by quadrics, leads us to ask the following question.

\begin{question}\label{Q2} 
 Let $G \subset \GL(3,\kk)$ be a finite abelian group of order $d$. Is it true that 
$\A(X_{2,d}^G)$ is quadratic if and only if $\I(X_{2,d}^G)$  has a quadratic Gr\"obner basis?
\end{question}

 Next, we consider families of finite abelian groups $G \subset \GL(3,\kk)$ for which we are able to find a quadratic Gr\"obner basis for the ideal $\I(X^G_{n,d})$. 

\begin{proposition}\label{Proposition: G basis 0,1,k}
 Let $G = \langle M_{d;0,1,k} \rangle \subset \GL(3,\kk)$ and $d = tk(k-1)$ for some integer $t \geq 1$.
 Then, $\I(X_{2,d}^G)$ has a quadratic Gr\"obner basis. 
\end{proposition}
\begin{proof} We denote $\cB_1 = \{m_0,\hdots, m_{\mu_{d}}\}$ and we take $S = \kk[w_0,\hdots,w_{\mu_d}]$. We want to prove that there is a term order $(S,\preceq)$ such that $\I(X_{2,d}^G)$ admits a Gr\"obner basis of quadrics.  To this end, we observe that a monomial $m = x_0^ax_1^bx_{2}^c$ belongs to $\cB_1$ if and only if there exists an integer $r$ with $0 \leq r \leq k$ such that
\[(*)_{r}: \left\{\begin{array}{lllllll}
a &+& b &+& c &=& d\\
    & & b &+& kc&=&rd. 
\end{array}\right.\]
The system $(*)_{0}$ has only one solution $(d,0,0)$, and for $r$ with $1 \leq r \leq k$, we have solutions
\begin{equation}\label{Eq: solutions}
\{((k-1)c - (r-1)d , \ rd-kc, \ c)\mid c = (r-1)tk,  \  \hdots \ , rt(k-1)\}
\end{equation}
determined by $r$ and $c$. 
We set $I_{0}:=\{0\}$ and $I_r:= \{(r-1)tk, \hdots, rt(k-1)\}$ for $1 \leq r \leq k$. Thus, $\cB_1$ is uniquely determined by the set $\W = \{(r,c) \mid  0 \leq r \leq k \; \text{and} \; c \in I_{r}\}$. We let $m_{(r,c)}$ denote the monomial in $\cB_1$ determined by $(r,c)$, and we write $(r_i,c_i)$ for the pair in $\W$ corresponding to the monomial $m_i \in \cB_1$. With this notation, a binomial $w^{\alpha} = w_{i_1}\!\cdots w_{i_{s}} - w_{j_1}\!\cdots w_{j_{s}}$ belongs to $\I(X_{2,d}^G)$ if and only if we have equalities
\begin{equation}\label{Eq:equalities_GB_01k}
\begin{array}{lll}
r_{i_1} + \dots + r_{i_s} & = & r_{j_1} + \dots + r_{j_s}\\
c_{i_1} + \dots + c_{i_{s}} & = & c_{j_1} + \dots + c_{j_{s}}.
\end{array}
\end{equation}
Using the description (\ref{Eq: solutions}), we sort the monomials in $\cB_1$ as follows. Given $m_{i},m_{j} \in \cB_1$, we say that $m_i < m_j$ if and only if $r_i < r_j$ or,  $r_i = r_j$ and $c_i < c_j$. Now in $S$, we say that $w_i \preceq w_j$ if and only if $m_i \leq m_j$ and we take $(S,\preceq)$ as Lex with $\{w_1,\hdots,w_{\mu_d}\}$ ordered by $\preceq$. Our goal is to prove that $\I(X_{2,d}^G)$ has a quadratic Gr\"obner basis with respect to $\preceq$. Let $w^{\alpha} \in \I(X_{2,d}^G)$ be a non-trivial $3$-binomial. To achieve our goal, it suffices to show that there is a $2$-binomial $w^{\beta} \in \I(X_{2,d}^G)$ such that $\LT(w^{\beta})$ divides $\LT(w^{\alpha})$. 

We put $w^{\alpha} = w_{i}w_jw_{\ell} - w_{t}w_{u}w_{v}$ with $\LT(w^{\alpha}) = w_iw_jw_\ell$ and we assume that $w_i \preceq w_j \preceq w_{\ell}$ and $w_t \preceq w_u \preceq w_v$. The hypothesis $w^{\alpha}$ non-trivial means that $\{w_i,w_j,w_\ell\} \cap \{w_{t},w_{u},w_{v}\} = \emptyset$, which implies $\LT(w^{\alpha}) = w_iw_jw_\ell$ if and only if $w_v \prec w_{\ell}$, i.\,e.\ $r_{v} < r_{\ell}$ or $r_v = r_\ell$  and $c_v < c_{\ell}.$ We will strongly use the following facts 
\begin{equation}\label{Eq: facts} \begin{array}{c}
|I_0| = 1, \quad |I_{r}| = t(k-r) + 1, \; 1 \leq r \leq k, \\
t(k-1)+1 = |I_{1}| > |I_{2}| > \cdots > |I_{k-1}| = t + 1 > |I_{k}| = 1,\\
\min I_r - \max I_{r-1} = (r-1)tk - (r-1)t(k-1) = (r-1)t.
\end{array}
\end{equation}

\noindent {\bf Claim.} Let $w\widetilde{w} \in S$ be a monomial of degree $2$. Assume 
\begin{enumerate}[(A)]
    \item $r = 0$ and $\widetilde{r} \geq  2$, or
    \item $r > 0$ and one of the following conditions holds.
    \begin{enumerate}[(i)]
        \item $r \leq \widetilde{r}-1$, $c < \max I_r=rt(k-1)$ and $\widetilde{c} > \min I_{\widetilde r}= (\widetilde{r}-1)tk$,
        \item $r \leq \widetilde{r}-2$,
        \item $r = \widetilde{r}$ and $c \leq \widetilde{c} - 2$.
    \end{enumerate}
\end{enumerate}
Then, there is a $2$-binomial $w^{\beta} \in \I(X_{2,d}^G)$ with $\LT(w^{\beta}) = w\tilde{w}$.

Assume that {\bf Claim} is true. We want to prove that there is a $2$-binomial $w^{\beta} \in \I(X_{2,d}^G)$ such that $\LT(w^{\beta})$ divides $w_{i}w_{j}w_{\ell}$. We distinguish several cases depending on the values of $(r_i,r_j,r_l)$. 

\vspace{0.15cm}
\noindent \underline{$r_i = 0$.} If $r_j \ge 2$ or $r_\ell \ge 2$ we can apply {\bf Claim}(A) with $w_iw_j$ or $w_iw_\ell$. 
If $r_j = 0$ and $r_\ell=0$ or $1$ the equalities (\ref{Eq:equalities_GB_01k}) have only one solution, so $w^\alpha$ would be trivial. If $r_j=r_\ell = 1$ the equalities (\ref{Eq:equalities_GB_01k}) imply $(r_t,r_u,r_v) = (0,1,1)$ and $c_i=c_t=0$, so $w^{\alpha}$ would again be trivial. 

\vspace{0.15cm}
\noindent \underline{$r_i > 0$.} {\bf Claim}(B) with $w_{i}w_{j}$ or $w_{i}w_{\ell}$ applies unless we are in one of the following three subcases.

\vspace{0.15cm}
\noindent \underline{$r_i = r_j = r_\ell$, $c_j = c_i + \delta_j$ and $c_\ell = c_i + \delta_{\ell}$ with $\delta_{j},\delta_{\ell} \in \{0,1\}$.} Then, (\ref{Eq:equalities_GB_01k}) and $\LT(w^{\alpha}) = w_iw_jw_{\ell}$ imply $r_t = r_u = r_v = r_i$ and $c_v < c_\ell \leq c_i + 1$. In particular, (\ref{Eq:equalities_GB_01k}) gives us $3c_i + \delta_{j} + \delta_{\ell} =  c_t + c_u + c_v$. Since $c_v < c_\ell$, we have $2c_i + \delta_{j} < c_t + c_u$ and, hence, $c_t > c_i$ or $c_u > c_i$ which is a contradiction because then we get $w_{v} \prec w_{t}$ or $w_{v} \prec w_{u}$. 

\vspace{0.15cm}
\noindent \underline{$r_i = r_j < r_\ell = r_i+1$ with $c_j = c_i + \delta_{j}$, $\delta_{j} \in \{0,1\}$, and $c_i = \max I_{r_i}$ or $c_\ell = \min I_{r_\ell}$.} As before, (\ref{Eq:equalities_GB_01k}) and the inequality $r_v \le r_\ell$ imply $(r_t,r_u,r_v) = (r_i,r_j,r_\ell)$. Hence we have $c_v < c_\ell$, and in particular $c_\ell > \min I_{r_\ell}$. Then $c_i = \max I_{r_i}$ and therefore $c_j = c_i$. Now by (\ref{Eq:equalities_GB_01k}) we obtain $2c_i < c_t + c_u$ which gives us $c_t > c_i$ or $c_u > c_i$. But this is a contradiction since $c_i = \max I_{r_{i}}$. 

\vspace{0.15cm}
\noindent \underline{$r_i < r_j = r_\ell = r_{i}+1$ with $c_\ell = c_j + \delta_{\ell}$, $\delta_{\ell} \in \{0,1\}$, and $c_i = \max I_{r_i}$ or $c_\ell = \min I_{r_\ell}$.} By (\ref{Eq:equalities_GB_01k}) we have $r_{t}+r_{u}+r_{v} = 3r_i + 2$ with $r_u \leq r_u \leq r_v \leq r_{i}+1$ which implies $(r_t,r_u,r_v) = (r_i,r_i+1,r_{i}+1)$. Since $w_{v} \prec w_{l}$, we have $c_v < c_\ell$ and in particular $c_\ell > \min I_{r_\ell}$. Hence $c_i = \max I_{r_i}$, and as $r_t=r_i$ we have $c_t \le c_i$. Now by (\ref{Eq:equalities_GB_01k}) we obtain $c_j < c_u$ which implies $c_\ell \le c_u$. But then we get $w_{l} \preceq w_{u} \preceq w_{v}$ which contradicts $w_{v} \prec w_{l}$.

To finish, it only remains to prove the {\bf Claim}. Let us start with (A) and assume $r = 0$ and $\tilde{r} \geq 2$. We write $\widetilde{c} = (\tilde{r}-1)tk + \delta$, with $\delta \leq t(k-\widetilde{r})$ so that $\widetilde{c} \in I_{\widetilde{r}}$. We define   
\[(r',c') =(1,(\widetilde{r}-1)t +\delta), \quad  \text{and} \quad (r'',c'')= (\widetilde{r}-1, (\tilde{r}-1)t(k-1)).\] 
Since $\delta \leq t(k-\widetilde{r})$ and $(\tilde{r}-1)t+\delta \leq t(k-1)$ by (\ref{Eq: facts}), we obtain $(r',c'),(r'',c'') \in \W$ and $m_{(r,c)}m_{(\widetilde{r},\widetilde{c})} = m_{(r',c')}m_{(r'',c'')}$. This induces a $2$-binomial $w^{\beta}$ with $\LT(w^{\beta}) = w\widetilde{w}$. 

We move on to (B) and assume $r > 0$. To prove (i) and (iii), we define $(r',c') = (r,c+1)$ and $(r'',c'') = (\tilde{r},\widetilde{c}-1)$. This defines a 2-binomial with leading term $w\widetilde w$ in the same way we saw above. The same construction can also be applied to prove (ii) except in the two subcases
\begin{enumerate}[(a)]
\item $(c,\widetilde{c}) = (rt(k-1), (\widetilde{r}-1)tk + \delta)$, 
\item $(c,\widetilde{c}) = (rt(k-1)-\delta, (\widetilde{r}-1)tk)$. 
\end{enumerate}
Both are analogues and we deal only with (a). We define
\[(r',c') = (r+1, (r+1)t(k-1)) \quad \text{and} \quad  (r'',c'') = (\widetilde{r}-1, (\widetilde{r}-2)tk+\delta+t).\]
By (\ref{Eq: facts}), $\delta \leq |I_{\tilde{r}}| = t(k-\tilde{r})$ and $|I_{\widetilde{r}-1}| = t(k - \widetilde{r}+1)$, so the pairs $(r',c'), (r'',c'') \in \W$. In all cases, we obtain $m_{(r,c)}m_{(\widetilde{r},\widetilde{c})} = m_{(r',c')}m_{(r'',c'')}$, which induces a $2$-binomial $w^{\beta}$ with $\LT(w^{\beta}) = w\widetilde{w}$. 
\end{proof}

\begin{remark}
 We have seen in  Proposition \ref{prop:G_koszul} that the algebra $R^{\overline{G}}$, with $G=\langle M_{d;0,1,k} \rangle \subset \GL(3,\kk)$, is Koszul if and only if $\GCD(k,d)\GCD(k-1,d)>1$. The term order considered in the proof of Proposition \ref{Proposition: G basis 0,1,k} above does not always provide a quadratic Gröbner basis when the condition $d=tk(k-1)$ is dropped. However, there are many other possible term orders.
We have verified it for $d \le 25$, by computation in Macaulay2 \cite{Macaulay2},  that each $\I(X^G_{n,d})$ has a quadratic Gröbner basis w.\,r.\,t.\ some term order.
\end{remark}

\begin{proposition}\label{prop:G-quadratic_group} Let $G \subset \GL(3,\kk)$ be a cyclic group of order $d$. In the cases
\begin{itemize}
    \item [(i)] $G=\langle M_{d;0,k,d-k} \rangle$ with $d$ even and $\GCD(d,k) = 1$,
    \item [(ii)] $G = \langle M_{d;0,\alpha_1, \alpha_2} \rangle$ with $\GCD(d,\alpha_1,\alpha_2) = \delta > 1$,
\end{itemize}
the ideal $\I(X^G_{2,d})$ has a quadratic Gröbner basis.
\end{proposition}
\begin{proof} (i) Since $\GCD(d,k) = 1$, we have $G = \langle M_{d;0,1,d-1} \rangle$. If $d = 2$, the result is trivial, so we take $d \geq 4$ even. We set $G' = \langle M_{d;0,1,2} \rangle$, a straightforward linear change of variable shows that $R^{G}$ and $R^{G'}$ are isomorphic. Now (i) follows from Proposition \ref{Proposition: G basis 0,1,k}. 

\noindent (ii) We set $d' = d/\delta, \,\alpha_1' = \alpha_1/\delta, \,\alpha_2' = \alpha_2/\delta$ and let $G' = \langle M_{d';0,\alpha_1',\alpha_2'} \rangle$. We have that the set of monomial invariants of $G$ of degree $d$ is the set of monomial invariants of $G'$ of degree $d = \delta d'$. Hence, $R^{\bG}$ is the $\delta$th Veronese subalgebra of $R^{\overline{G'}}$. Since $R^{\overline{G'}}$ has Castelnuovo-Mumford regularity $\reg(R^{\overline{G'}}) \leq 3$, by \cite[Theorem 2]{ERT} the ideal $\I(X^G_{2,d})$ has a quadratic Gr\"obner basis. 
\end{proof}

As in the proof of Proposition \ref{prop:G-quadratic_group} above, we can apply  \cite[Theorem 2]{ERT} to assure that, for any finite abelian group $G \subset \GL(n+1,\kk)$ of order $d$ and any integer $t \geq \lceil \frac{\reg(\A(X_{n,d}^G))}{2} \rceil$, the ring $R^{\bG^t}$ of invariants of the $t$-th cyclic extension of $G$ is a quadratic Gr\"obner algebra. 

\begin{problem} Given a finite abelian group $G \subset \GL(n+1,\kk)$ of order $d$, which is the smallest integer $t \geq 1$ such that $R^{\bG^t}$ is a quadratic Gr\"obner algebra? 
\end{problem}

In particular, for any finite abelian group $G \subset \GL(3,\kk)$, we have $t \leq 2$ and this bound is sharp by Proposition \ref{Proposition: degree binomial GT-surfaces}. 

The next lemma allows us to produce quadratic Gr\"obner algebras $\A(X_{n,d}^G)$ in any dimension $n \geq 2$. 
For some positive integers $\ell_1, \ldots, \ell_n$, we consider a polynomial ring $\widetilde{R} = \kk[x_{ij}]$ on $N=\ell_1+\ldots + \ell_n$ variables $x_{ij}$, where $1 \le i \le n$ and $1 \le j \le \ell_i$. We define a homomorphism $\psi: \widetilde{R} \longrightarrow R$ by $x_{ij} \mapsto x_i$.   

\begin{lemma}\label{lemma:GB_extension}
Let $\Omega_{n,d} \subseteq \cM_{n,d}$, and let $\Omega_{N,d} \subset \widetilde{R}$ be the preimage of $\Omega_{n,d}$ under $\psi$ as defined above. Take $Y_{n,d}$ and $Y_{N,d}$ to be the monomial projections parameterized by $\Omega_{n,d}$ and $\Omega_{N,d}$. If the ideal $\I(Y_{n,d})$ is generated in degrees at most $k$, so is $\I(Y_{N,d})$. Moreover, if $\I(Y_{n,d})$ has a Gröbner basis of binomials of degrees at most $k$, so does $\I(Y_{N,d})$.
\end{lemma}
\begin{proof}
It is enough to consider $\widetilde{R}=\kk[x_{01},x_{02},x_1, \ldots, x_n]$, as the general statement then follows by repeated use of the argument.

We will use the notation $x^\alpha$ and $x^{\widetilde \alpha}$ for monomials in $R$ and $\widetilde R$ with exponent vectors $\alpha$ and $\widetilde \alpha$. Let $S=\kk[w_\alpha]$ be the ambient ring of $\I(Y_{n,d})$, here on variables indexed by all $\alpha$ such that $x^\alpha \in \Omega$. Similarly, we define $\widetilde S = \kk[w_{\widetilde \alpha}]$. Let $\varphi:\widetilde S \longrightarrow S$ be the homomorphism defined by 
\[
\varphi(w_{(\alpha_{01}, \alpha_{02}, \alpha_1, \ldots, \alpha_n)}) = w_{(\alpha_{01}+ \alpha_{02}, \alpha_1, \ldots, \alpha_n)}.
\]
It is clear that the preimage of a generating set of $\I(Y_{n,d})$ under $\varphi$ is a generating set of $\I(Y_{N,d})$. 

Assume $\I(Y_{n,d})$ has a Gröbner basis of binomials of degree $\le k$ w.\,r.\,t.\ a term order $\succ$ on $S$. We extend $\succ$ to an order of the variables of $\widetilde S$ by declaring $w_{\widetilde \alpha} \succ w_{\widetilde \beta}$ if $\varphi(w_{\widetilde \alpha})\succ \varphi(w_{\widetilde \beta})$ or $\varphi(w_{\widetilde \alpha})=\varphi(w_{\widetilde \beta})$ and $\alpha_{01}>\beta_{01}$. Now for monomials $u,v$ in $\widetilde S$, we say that $u>v$ if $\varphi(u)\succ \varphi(v)$ or $\varphi(u)=\varphi(v)$ and $u\succ v$ according to RevLex on $\widetilde S$. It is easily verified that this is a proper term order on $\bar{S}$. 

To prove that $\I(Y_{N,d})$ has a degree $k$ Gröbner basis w.\,r.\,t.\ this order, take a degree $r$ binomial
\[f=w_{\widetilde\alpha^{(1)}} \cdots w_{\widetilde\alpha^{(r)}}-w_{\widetilde\beta^{(1)}} \cdots w_{\widetilde\beta^{(r)}} \in \I(Y_{N,d}),\]
 where $r>k$ and $w_{\widetilde\alpha^{(1)}} \cdots w_{\widetilde\alpha^{(r)}}$ is the leading term. We want to prove that this term is divisible by the leading term of some binomial of degree $\le k$ in $\I(Y_{N,d})$. We assume that the two terms in $f$ have no common factor. The proof is carried out in two cases. 

\noindent \underline{Case 1:} $\varphi(f) \ne 0$. Then, $\varphi(f)$ is a non-trivial binomial in $\I(Y_{n,d})$. It follows from how we defined the ordering that $\varphi(w_{\widetilde\alpha^{(1)}} \cdots w_{\widetilde\alpha^{(k)}})=w_{\alpha^{(1)}} \cdots w_{\alpha^{(k)}}$ is the leading term of $\varphi(f)$. By assumption, $w_{\alpha^{(1)}} \cdots w_{\alpha^{(k)}}$ is divisible by the leading term of some binomial of degree $s\le k$ in $I$. Say $w_{\alpha^{(1)}}\cdots w_{\alpha^{(s)}}-w_{\gamma^{(1)}}\cdots w_{\gamma^{(s)}} \in \I(Y_{n,d})$, with leading term $w_{\alpha^{(1)}}\cdots w_{\alpha^{(s)}}$. Then, we lift $w_{\gamma^{(1)}}\cdots w_{\gamma^{(s)}} \in S$ to $w_{\widetilde \gamma^{(1)}}\cdots w_{\widetilde \gamma^{(s)}} \in \widetilde S$ so that $w_{\widetilde\alpha^{(1)}} \cdots w_{\widetilde\alpha^{(s)}}-w_{\widetilde\gamma^{(1)}} \cdots w_{\widetilde\gamma^{(s)}} \in \I(Y_{N,d})$. 
The choice of $\gamma_{01}^{(1)}, \gamma_{02}^{(1)}, \ldots,  \gamma_{01}^{(s)}, \gamma_{02}^{(s)}$ is not unique, but 
as $w_{\alpha^{(1)}} \cdots w_{\alpha^{(s)}}\succ w_{\gamma^{(1)}}\cdots w_{\gamma^{(s)}}$, we are guaranteed that $w_{\widetilde\alpha^{(1)}}\cdots w_{\widetilde\alpha^{(s)}}\succ w_{\widetilde\gamma^{(1)}} \cdots w_{\widetilde\gamma^{(s)}}$.

\noindent \underline{Case 2.} $\varphi(f)=0$. In this case, we turn to the RevLex order on $\widetilde S$.
We can assume we have numbered the factors so that $w_{\widetilde\alpha^{(1)}}\succ \cdots \succ w_{\widetilde\alpha^{(r)}}$ and $w_{\widetilde\beta^{(1)}}\succ \cdots \succ w_{\widetilde\beta^{(r)}}$. 
As $\varphi$ respects the order and $\varphi(f)=0$,  we have $\varphi(w_{\widetilde\alpha^{(i)}})=\varphi(w_{\widetilde\beta^{(i)}}) $.
Since $w_{\widetilde\alpha^{(1)}} \cdots w_{\widetilde\alpha^{(r)}}$ is leading w.\,r.\,t.\ RevLex,  this tells us that $w_{\widetilde\alpha^{(r)}}\succ w_{\widetilde\beta^{(r)}}$. Then $\alpha_{01}^{(r)}>\beta_{01}^{(r)} \ge 0$. Thus, in some other factor $w_{\widetilde \alpha^{(i)}}$, we have $\alpha_{02}^{(i)}>0$. From this, we construct the relation
\[
w_{\widetilde \alpha^{(i)}}w_{\widetilde \alpha^{(r)}} - w_{(\alpha_{01}^{(i)}+1,\alpha_{02}^{(i)}-1,\alpha_1^{(i)}, \ldots, \alpha_n^{(i)})}w_{(\alpha_{01}^{(r)}-1,\alpha_{02}^{(r)}+1,\alpha_1^{(r)}, \ldots, \alpha_n^{(r)})} \in \I(Y_{N,d}).
\]
As
\[
w_{(\alpha_{01}^{(i)}+1,\alpha_{02}^{(i)}-1,\alpha_1^{(i)}, \ldots, \alpha_n^{(i)})} \succ w_{\widetilde \alpha^{(i)}}\succeq w_{\widetilde \alpha^{(r)}}\succ w_{(\alpha_{01}^{(r)}-1,\alpha_{02}^{(r)}+1,\alpha_1^{(r)}, \ldots, \alpha_n^{(r)})}, 
\]
the leading term is $w_{\widetilde \alpha^{(i)}}w_{\widetilde \alpha^{(r)}}$. 
\end{proof}

Proposition \ref{prop:GB_extend_groups} below illustrates how to apply Lemma \ref{lemma:GB_extension} to algebras $R^G$. To avoid heavy notation, we state the proposition for cyclic groups. The construction can then be applied to any abelian group.

\begin{proposition}\label{prop:GB_extend_groups}
Let $G=\langle M_{d;0,\alpha_1, \ldots, \alpha_n} \rangle \subset \GL(n+1,\kk)$, and let $\widetilde{M}$ be the diagonal block matrix with blocks $M_{d;0, \ldots, 0}, M_{d; \alpha_1 \ldots, \alpha_1}$, \ldots, $M_{d;\alpha_n, \ldots, \alpha_n}$ of any sizes. 
If $\A(X_{n,d}^G)$ is a quadratic Gr\"obner algebra, so is the algebra given by the group $\langle \widetilde{M}  \rangle $. 
\end{proposition}
\begin{proof}
We take $\ell_1, \ldots, \ell_n$ to be the sizes of the diagonal blocks of $\widetilde{M}$, and define a map $\psi$ as in Lemma \ref{lemma:GB_extension}. Then, the set of monomial invariants of $G$ is the image of the monomial invariants of $\langle \widetilde M \rangle$ under $\psi$. The result then follows from Lemma \ref{lemma:GB_extension}.
\end{proof}

Applying Proposition \ref{prop:GB_extend_groups} to the groups in Proposition \ref{prop:G-quadratic_group}, we get the following corollary.

\begin{corollary}\label{coro:GB_extend_groups}
Let $G=\langle M \rangle$ be a cyclic group of order $d$, where $M$ is a diagonal block matrix with three blocks $M_{d;0, \ldots, 0}, M_{d; k, \ldots, k}, M_{d;d-k, \ldots, d-k}$ of any sizes. If $d$ is even and $\GCD(d,k)=1$, then the algebra $\A(X_{n,d}^G)$ is G-quadratic. The same holds if $M$ is composed of three blocks  $M_{d;0, \ldots, 0}, M_{d; \alpha_1, \ldots, \alpha_1}, M_{d;\alpha_2, \ldots, \alpha_2}$ and $\GCD(d,\alpha_1,\alpha_2)>1.$
\end{corollary}

Example \ref{ex:Koszul_group_action} explores an alternative approach to tackle the Koszulness of algebras  $\A(X_{n,d}^G)$.

\begin{example}\label{ex:Koszul_group_action}
Let $G = \langle M_{4;0,1,2,3} \rangle \subset \GL(4,\kk)$ be a cyclic group of order $4$. 
The algebra $R^{\overline{G}}$ is generated by 
the set of monomials $\cB_1$ given in Example \ref{ex:group_inv_alg}. 
The defining ideal $\I(X_{3,4}^G) \subset \kk[w_0, \ldots, w_{9}]$ has quadratic Gröbner basis under the RevLex term order with $w_0> \dots >w_{9}$, if the variables $w_i$ are mapped to the monomials in $\cB_1$ sorted according to RevLex with $x_1 > x_3 > x_0 > x_2$. An alternative way to prove that $R^{\overline{G}}$ is Koszul is the following. Let $S$ be the Segre product of two copies of the second Veronese subring on three variables. More concretely, we let
\[
S=\kk[y_0^2, y_1^2, y_2^2, y_0y_1, y_0y_2, y_1y_2] \circ \kk[z_0^2, z_1^2, z_2^2, z_0z_1, z_0z_2, z_1z_2] = \kk[y_iy_jz_kz_\ell]_{0 \le i,j,k,\ell \le 2}.
\]
This is a Koszul algebra as Segre products of Veronese subalgebras are Koszul, by \cite{Backelin-Froberg}. A surjective homomorphism $S \longrightarrow R^{\overline{G}}$ is induced by 
\[
y_0 \mapsto x_0^2, \quad y_1 \mapsto x_2^2, \quad y_2 \mapsto x_1x_3, \quad z_0 \mapsto x_0x_2, \quad z_1 \mapsto x_1^2, \quad z_2 \mapsto x_3^2.
\]
In fact $R^{\overline{G}} \cong S/ (y_0y_1 - z_0^2, y_2^2-z_1z_2)$, and $y_0y_1 - z_0^2, y_2^2-z_1z_2$ is a regular sequence of linear forms on the generators of $S$. This proves that $R^{\overline{G}}$ is Koszul, as the quotient of a Koszul algebra by a regular sequence of elements of degree one or two is again Koszul, by \cite{Barcanescu-Manolache}.
\end{example}

Unfortunately, we were unable to use the technique from Example \ref{ex:Koszul_group_action} on $\langle M_{6;0,1,2,3} \rangle$, which is also quadratic by Proposition \ref{Proposition: degree binomial GT-threefold}. Nevertheless, we are optimistic that there are other examples where similar ideas can be applied.

\section{Open problems and conjectures}
\label{Section: conjectures}
In this last section, we collect the open questions encountered in the preceding
sections, and state two conjectures.

\begin{enumerate}

\item Let $Y_{2,4}$ be the monomial projection  parameterized by $\cM_{2,4} \setminus \{x_0x_1x_2^2\}$. Is $\A(Y_{2,4})$ a Koszul algebra?.

\item Are the pinched Veronese algebras $\PV(n,d,s)$ with $d>s>1$ quadratic only if $s \ge \lceil \frac{n+2}{2} \rceil$, say with a few sporadic exceptions?.

\item Given a finite abelian group $G \subset \GL(n+1,\kk)$ of order $d$, can we determine if $\A(X_{n,d}^G)$ is a quadratic algebra?

\item  Let  $G \subset \GL(3,\kk)$ be a finite abelian group of order $d$. Is $\A(X_{2,d}^G)$ quadratic if and only if $\I(X_{2,d}^G)$ has a quadratic Gröbner basis?

\item Let $G=\langle M_{d;0,1,2,3} \rangle \subset \GL(4,\kk)$ be a cyclic group of order $d$. Assume that $d$ is even. Is $\A(X_{3,d}^G)$ quadratic if and only if $\I(X_{3,d}^G)$ has a quadratic Gröbner basis?
\end{enumerate}

We end the paper writing down two conjectures based on the results of the previous sections, on \cite[Corollary 5.7(2)]{CM-R}, as well as on Macaulay2 \cite{Macaulay2} computations.

\begin{conjecture}\label{Conj0}  Let $n\ge 3$ and take a cyclic group $G\subset \GL(n+1,\kk)$ of order $d$ generated by a diagonal matrix
\[
\left(\begin{array}{lllllll}
e^{\alpha_0} & 0 & \cdots & 0\\
0                & e^{\alpha_1} & \cdots & 0\\
\vdots & \vdots & \ddots & \vdots\\
0 & 0 & \cdots & e^{\alpha_n}\\
\end{array}\right)
\]
where $e$ is a $d$-th primitive root of $1\in \kk$. The algebra $\A(X_{n,d}^G)$ is quadratic if and only if for all $i<j<k$, the algebra $\A(X_{2,d}^{G^{i,j,k}})$ is quadratic, where $G^{i,j,k} \subset \GL(3,\kk)$ denotes the cyclic group generated by the diagonal matrix 
\[\left(\begin{array}{lllllll}
e^{\alpha_i} & 0 &  0\\
0                & e^{\alpha_j} &  0\\
0 & 0 & e^{\alpha_k}\\
\end{array}\right).
\]
\end{conjecture}

\begin{conjecture}\label{Conj2} Let $G \subset \GL(n+1,\kk)$ be a finite abelian group of order $d$. The algebra $\A(X_{n,d}^G)$ is quadratic if and only if $\I(X_{n,d}^G)$ has a quadratic Gröbner basis.
\end{conjecture}

\end{document}